\documentclass[final,hidelinks]{siamart}
\makeatletter
\AddToHook{begindocument/end}{%
  \def\refstepcounter@optarg[#1]#2{%
    \cref@old@refstepcounter{#2}%
    \cref@constructprefix{#2}{\cref@result}%
    \@ifundefined{cref@#1@alias}%
      {\def\@tempa{#1}}%
      {\def\@tempa{\csname cref@#1@alias\endcsname}}%
    \protected@edef\cref@currentlabel{%
      [\@tempa][\arabic{#2}][\cref@result]%
      \csname p@#2\endcsname\csname the#2\endcsname}}}
\makeatother
\usepackage{amssymb,amsfonts}
\newsiamremark{remark}{Remark}
\numberwithin{theorem}{section}
\numberwithin{equation}{section}
\numberwithin{figure}{section}
\numberwithin{table}{section}
\usepackage{tikz}
\usepackage{microtype}
\newcommand{\M}{{\mathcal M}}   
\newcommand{\Diag}{\mathrm{Diag}}

\begin{document}
\pagestyle{myheadings}
\markboth{X. LIU}{INVERSE SCATTERING FROM BACKSCATTERING DATA}

\title{The Interior Inverse Scattering Problem from Backscattering Data: Spectral Comparison and Uniqueness Results}
\author{Xiaodong Liu\thanks{State Key Laboratory of Mathematical Sciences, Academy of Mathematics and Systems Science, Chinese Academy of Sciences, Beijing 100190, China (\texttt{xdliu@amt.ac.cn}).}}
\date{}
\maketitle

\begin{abstract}
We study an interior inverse scattering problem for the two-dimensional
Helmholtz equation with backscattering data, where the source and receiver
coincide. Such data are common in applications, yet existing uniqueness
results are almost exclusively established for full multistatic data with separated sources and receivers. We consider two main measurement regimes.
In the first regime, data are collected on a closed curve inside the unknown domain. Exploiting the meromorphic dependence of the backscattering field
on the wavenumber, we show that the residues at the poles determine the
diagonal of the spectral projection kernel on the curve, and hence the
\emph{visible} part of the Dirichlet spectrum. This yields an unconditional
spectral comparison: the symmetric difference of the two Dirichlet spectra
is contained in the spectrum of the interior domain enclosed by the curve.
If, in addition, the common visible eigenvalues have either matching
multiplicities or matching restriction ranks, we obtain an area
estimate bounding the discrepancy of the two domain areas by the area
of this interior domain. In the second regime,
data are given on a nonempty open subset of the domain, and a single
non-resonant frequency already uniquely determines a bounded connected Lipschitz
domain. The proof combines interior real-analyticity of the
regular part of the Helmholtz Green function with its boundary blow-up.
\end{abstract}



\section{Introduction}
The inverse scattering problem asks whether the geometry of an unknown domain can be recovered from scattered-field data. Classical uniqueness results typically rely on much richer information, such as the full far-field pattern or the scattered field for independent source and receiver locations. Comprehensive accounts of the underlying theory and numerical methods can be found in the monographs of Colton and Kress \cite{ColtonKress2019} and Isakov \cite{Isakov2006}.
By contrast, the backscattering setting, in which source and receiver coincide, is more restrictive, and the corresponding uniqueness theory is far less developed.

Throughout, \emph{domain} denotes the unknown set. The term \emph{cavity} is used when the interior configuration is emphasized, while
\emph{obstacle} is reserved for the exterior scattering literature cited
below.

For obstacle scattering under reduced backscattering data, the existing literature focuses largely on special geometries and measurement configurations. Kress and Rundell \cite{KressRun98} proved local uniqueness near the unit disk by analyzing Fr\'echet derivatives. Using high-frequency asymptotics of the far field, Christiansen \cite{Christiansen2013} proved that a strictly convex real-analytic obstacle with Dirichlet or Neumann boundary condition is determined by generalized backscattering data in two directions. Shin \cite{Shin2016} showed that the radius of a sound-soft ball can be uniquely recovered from phaseless backscattering data. The recent work of Li, Liu and Shi \cite{LiLiuShi2026} and of Li and
Liu \cite{LiLiuNearfield} established uniqueness for convex obstacles
from multi-frequency backscattering data in the exterior problem using
high-frequency asymptotic expansions. The present paper concerns diagonal Green-function data
from interior backscattering measurements.

In the interior cavity setting, Qin and Colton \cite{QinColton2012,QinColton2012acm}
initiated the study of inverse problems with sources and measurements
placed on an interior curve, but with data sets richer than the backscattering values considered here. Qin and Liu \cite{QinLiu2015} also placed sources and receivers on an
interior curve and proved uniqueness for cavities after introducing an
artificial impedance obstacle. Hu and Liu \cite{HuLiu2014} proved uniqueness in the interior setting for disks and polygons from the scattered field generated by a single fixed source, using a reflection-principle argument tied to those special geometries. Related reciprocity-gap approaches to interior inverse scattering use Cauchy data on an interior observation curve, that is, both the trace and the normal derivative of the total field are prescribed on that curve; see, e.g., Zeng, Liu, Sun, and Xu \cite{ZengLiuSunXu2017}. Those data are much richer than the coincident diagonal values \(y\mapsto u^s(y,y,k)\) considered here.  

The present work is also related to inverse problems from local spectral and heat data. Bosi, Kurylev, and Lassas \cite{BosiKurylevLassas2022} studied Gel'fand's inverse interior spectral problem using eigenvalues together with eigenfunctions restricted to an open set, while Kurylev, Lassas, Lu, and Yamaguchi \cite{KurylevLassasLuYamaguchi2012} considered recovery from heat-kernel measurements on an open subset. However, our single-frequency backscattering data on an open set are much weaker. They do not include the full heat kernel, the restrictions of all eigenfunctions to an interior open set, or even the spectral measure. Rather, they provide only the regular part of the diagonal of the Helmholtz Green function at one non-resonant frequency. We show that this much
weaker datum already determines the domain uniquely.

We now formulate the interior problem. Let \(D\subset\mathbb R^2\) be a
bounded connected Lipschitz domain. For a point source at $y \in D$, the incident field is
\begin{equation}
u^{\mathrm{inc}}(x,y,k)=\Phi_k(x,y):=\frac{i}{4}H_0^{(1)}(k|x-y|), \quad y\in D, \quad x\neq y,
\end{equation}
where $H_0^{(1)}$ denotes the Hankel function of the first kind of order zero, $k>0$ is the wavenumber. Assuming that $k^2\notin \sigma(D)$, where $\sigma(D)$ denotes the Dirichlet spectrum of $-\Delta$ in $D$, the corresponding scattered field $u^s(\cdot,y,k)\in H^{1}(D)$ satisfies
\begin{align}
(\Delta + k^2)\, u^s &= 0 \quad \text{in } D, \label{eq:helmholtz}\\
u^s &= -\Phi_k(\cdot,y) \quad \text{on } \partial D. \label{eq:dirichlet}
\end{align}

The inverse problem is to recover information about the domain $D$ from one of the following two backscattering data sets:
\begin{itemize}
  \item {\bf multi-frequency curve measurement}
  \begin{equation}
\bigl\{u^s(y,y,k): \, y\in \M, \, k_{\min}<k<k_{\max}, \, k^2\notin \sigma(D)\bigr\},
\end{equation}
where \(\M\Subset D\) is a closed Lipschitz curve and \(0<k_{\min}<k_{\max}\) are fixed constants.
  \item {\bf single-frequency open-set measurement}
   \begin{equation}
\{u^s(y,y,k_0): \, y\in \mathcal U, \, k_0^2\notin \sigma(D)\},
\end{equation}
where $\mathcal U$ is a nonempty open subset of $D$.
\end{itemize}

In the curve-measurement setting, we exploit the meromorphic dependence
of \(u^s(y,y,k)\) on the wavenumber. This spectral
approach yields the diagonal spectral projection kernel, and hence the
visible spectral measure on the measurement curve,
naturally distinguishing \emph{visible} from \emph{invisible} eigenvalues
according to whether the diagonal residue on the curve vanishes. It identifies only diagonal spectral data on the measurement curve, so
the area estimate requires an additional spectral assumption, not
supplied by the residue argument; see Section~\ref{sec:discussion}. Let \(\Omega\) denote the bounded connected Lipschitz domain enclosed by \(\M\), so that \(\M=\partial\Omega\). If the backscattering data of two candidate domains \(D_1\) and \(D_2\) coincide on \(\M\), then
the symmetric difference of the two spectra is contained in the spectrum of the enclosed domain:
\[
\sigma(D_1)\triangle\sigma(D_2)\subset\sigma(\Omega).
\]
Under the condition that every visible eigenvalue has the same multiplicity in \(D_1\)
and \(D_2\), we have
\[
\bigl||D_1|-|D_2|\bigr|\le|\Omega|,
\]
where \(|\cdot|\) denotes area. This is the conditional area estimate in Theorem~\ref{thm:main-unique}(b1). The same estimate also holds if the
restriction operators \(R_\lambda\) defined in Section~\ref{sec:uniqueness}
have equal ranks at every common visible eigenvalue.

In the open-set setting, we take a different route and do not extract
spectral data. At a fixed non-resonant frequency the measured quantity
is the Helmholtz diagonal regular part, whose interior analyticity and
boundary blow-up uniquely determine the domain. Thus
Theorem~\ref{thm:openset-single-frequency} requires only one frequency.
The argument consists of three steps: the measurement identity
converts coincident interior backscattering into the diagonal regular
part of the Dirichlet Green function; this diagonal quantity is
real-analytic in the interior and diverges at the boundary; and its
equality on a nonempty open subset forces equality of the domains. The boundary blow-up itself is classical in potential theory for diagonal Green
functions, Robin functions, and harmonic radii
\cite{DurenSchiffer1991,BandleFlucher1996,Flucher1999,GustafssonSakai1994,Sakai1982}. What is new is the inverse conclusion under the present measurement model: coincident interior backscattering identifies
the Helmholtz diagonal regular part on the observation set, and this leads to
a Robin-type inverse-domain theorem for bounded Lipschitz domains at nonzero
frequency. We are not aware of a domain-uniqueness result under precisely
this backscattering measurement model.

The remainder of this paper is organized as follows. Section~\ref{sec:spectral} presents the spectral tools used throughout: properties of the Dirichlet heat kernel, the meromorphic structure of the Green function with its residue formula, and the residue formula for diagonal spectral data. Section~\ref{sec:uniqueness} contains the main results. We prove the spectral comparison theorem, the area estimates under the
two alternative sufficient conditions, and the single-frequency
uniqueness theorem from open-set measurements. Section~\ref{sec:discussion} addresses the obstruction due to multiplicities, its relation to inverse boundary spectral problems, and the open problems left by this work. Concluding remarks are given in Section~\ref{sec:conclusion}.

\section{Preliminaries and Spectral Representation} \label{sec:spectral}

\subsection{Preliminaries}

We begin by fixing notation. Let \(D\subset\mathbb R^2\) be a bounded connected Lipschitz domain. We write $\Diag_{x=y}:=\{(x,y)\in D\times D:x=y\}$ for the diagonal of
$D\times D$. Let
\(A:=-\Delta\) denote the Dirichlet Laplacian on \(L^2(D)\), with domain
\[
\mathcal{D}(A) = \{u\in H^1_0(D) : \Delta u\in L^2(D)\},
\]
where the Laplacian is taken in the distributional sense. Then \(A\)
is self-adjoint, positive definite, and has compact resolvent, so its
spectrum is purely discrete \cite{Davies1989}. We write $\sigma(D)$ for the spectrum of $A$, that is, for the set of Dirichlet eigenvalues. We also set $\mathcal D(A^0):=L^2(D)$ and, for $m\ge1$, let $\mathcal D(A^m)$ denote the
domain of $A^m$, i.e. $\{u\in\mathcal D(A^{m-1}):A^{m-1}u\in\mathcal D(A)\}$.

Because the Dirichlet Laplacian has real coefficients and is self-adjoint, the eigenvalues  $\lambda\in\sigma(D)$ are positive and the corresponding eigenfunctions may be chosen to be \emph{real-valued}. For each $\lambda\in\sigma(D)$, let $m_\lambda$ denote its multiplicity, and let $\{\phi_{\lambda,j}\}_{j=1}^{m_\lambda}$ be an orthonormal basis of the eigenspace $\ker(A-\lambda)$. Thus each eigenfunction satisfies
\begin{equation*}
-\Delta \phi_{\lambda,j} = \lambda \phi_{\lambda,j} \quad \text{in } D, \quad \phi_{\lambda,j}|_{\partial D} = 0, \quad \|\phi_{\lambda,j}\|_{L^2(D)} = 1.
\end{equation*}

For each $\lambda\in\sigma(D)$, let $P_\lambda$ denote the orthogonal projection onto $\ker(A-\lambda)$. Its Schwartz kernel, denoted by the same symbol, is
\[
P_\lambda(x,y)=\sum_{j=1}^{m_\lambda}\phi_{\lambda,j}(x)\,\phi_{\lambda,j}(y),\quad x,y\in D,
\]
where $\{\phi_{\lambda,1},\dots,\phi_{\lambda,m_\lambda}\}$ is any orthonormal basis of $\ker(A-\lambda)$. For $y \in D$, we define the spectral measure
\begin{equation}\label{eq:spectral-measure}
\mu_y := \sum_{\lambda\in\sigma(D)} P_\lambda(y,y)\,\delta_\lambda,
\end{equation}
where $\delta_\lambda$ denotes the Dirac measure at $\lambda$.

For \(t>0\), the Dirichlet heat semigroup generated by \(A\) is denoted by \(e^{-tA}\) with its integral kernel \(K_D(t,x,y)\) defined as
\[
(e^{-tA}f)(x)=\int_D K_D(t,x,y)f(y)\,dy,
\quad f\in L^2(D).
\]
We call \(K_D\) the Dirichlet heat kernel of \(D\). The Dirichlet heat kernel has several well-known regularity and mapping properties that we collect here for later reference.

\begin{lemma}[Dirichlet heat kernel: bounds and vanishing]
\label{lem:heat-kernel-boundary}
Let
\begin{equation}\label{eq:Cpartial}
C_\partial(\overline D):=\{u\in C(\overline D):u|_{\partial D}=0\}.
\end{equation}
Then the following hold.
\begin{enumerate}
\item[(i)] There exist constants $C_0,b>0$, depending only on $D$, such that
\begin{equation}\label{eq:gaussian-upper-bound}
0\le K_D(t,x,y)\le \frac{C_0}{t}\,e^{-b|x-y|^2/t},
\quad t>0,\quad x,y\in D.
\end{equation}
\item[(ii)] There exist constants $C_1,C_2>0$, depending only on $D$, such that
for every $f\in L^2(D)$
\begin{equation}\label{eq:small-t-estimate}
\|e^{-tA}f\|_{L^\infty(D)}\le C_1\,t^{-1/2}\|f\|_{L^2(D)},
\quad 0<t\le 1,
\end{equation}
\begin{equation}\label{eq:large-t-estimate}
\|e^{-tA}f\|_{L^\infty(D)}\le C_2\,e^{-\lambda_1 t}\|f\|_{L^2(D)},
\quad t>1,
\end{equation}
where \(\lambda_1:=\min\sigma(D)\).
\item[(iii)] For every $t>0$, $e^{-tA}$ maps $L^2(D)$ boundedly into
$C_\partial(\overline D)$. The Dirichlet heat kernel admits a continuous
extension $K_D(t,\cdot,\cdot)\in C(\overline D\times \overline D)$ and
\begin{equation}\label{eq:kernel-vanishing}
K_D(t,x,y)=0\quad\text{whenever }x\in\partial D\ \text{or}\ y\in\partial D.
\end{equation}
\item[(iv)] For every \(y\in D\), the map
\((t,x)\mapsto K_D(t,x,y)\) is continuous on
\((0,\infty)\times\overline D\); by symmetry the same holds with the
roles of \(x\) and \(y\) interchanged.
\end{enumerate}
\end{lemma}
\begin{proof}
We prove (i)--(iv) in turn.

(i). Since $D\subset\mathbb R^2$, domain
monotonicity for the Dirichlet heat kernel gives \cite{vandenBerg1989}
\[
  0\le K_D(t,x,y)\le K_{\mathbb R^2}(t,x,y),
\]
where
\[
  K_{\mathbb R^2}(t,x,y)=(4\pi t)^{-1}e^{-|x-y|^2/(4t)}.
\]
Hence
\eqref{eq:gaussian-upper-bound} holds with $C_0=(4\pi)^{-1}$ and $b=1/4$.

(ii). By Cauchy--Schwarz and
\eqref{eq:gaussian-upper-bound}, for $0<t\le1$ and $f\in L^2(D)$,
\[
|e^{-tA}f(x)|\le\|K_D(t,x,\cdot)\|_{L^2(D)}\|f\|_{L^2(D)}
\le C_1 t^{-1/2}\|f\|_{L^2(D)},
\quad C_1:=C_0\sqrt{\frac{\pi}{2b}},
\]
which is \eqref{eq:small-t-estimate}. For $t>1$, writing
$e^{-tA}=e^{-A}e^{-(t-1)A}$ and applying \eqref{eq:small-t-estimate} at $t=1$
to the vector $e^{-(t-1)A}f$ yield
\[
\begin{aligned}
\|e^{-tA}f\|_{L^\infty(D)}
&=\|e^{-A}e^{-(t-1)A}f\|_{L^\infty(D)}
\le C_1\|e^{-(t-1)A}f\|_{L^2(D)}\\
&\le C_1e^{-\lambda_1(t-1)}\|f\|_{L^2(D)},
\end{aligned}
\]
which is \eqref{eq:large-t-estimate} with $C_2=C_1e^{\lambda_1}$. Here, in the last inequality, we have used the spectral theorem.

(iii). Since $D$ is a bounded Lipschitz domain, every boundary point of $D$ is
Wiener regular (see, e.g., \cite{Kenig1994}). Then Example~6.2 of Arendt--ter Elst \cite{ArendtElst2019} establishes the inclusion
\begin{equation}\label{eq:ATE-inclusion}
\{u\in \mathcal{D}(A):\ Au\in L^\infty(D)\}\subset C_\partial(\overline D).
\end{equation}
Now fix $t>0$ and $f\in L^2(D)$. Since $A$ commutes with $e^{-tA}$ and
$\|Ae^{-sA}\|_{\mathcal L(L^2(D))}\le(es)^{-1}$ (by the spectral theorem,
$\|Ae^{-sA}\|=\sup_{\lambda\in\sigma(D)}\lambda e^{-s\lambda}$, and
$\sup_{\lambda>0}\lambda e^{-s\lambda}=(es)^{-1}$, attained at
$\lambda=1/s$),
\[
Ae^{-tA}f=e^{-tA/2}\bigl(Ae^{-tA/2}f\bigr),
\quad
\|Ae^{-tA/2}f\|_{L^2(D)}\le\frac{2}{et}\|f\|_{L^2(D)}.
\]
By (ii),  $e^{-tA/2}$ maps $L^2(D)$ into $L^\infty(D)$; hence
$Ae^{-tA}f\in L^\infty(D)$, while $e^{-tA}f\in \mathcal{D}(A)$. Therefore
$e^{-tA}f\in C_\partial(\overline D)$ by \eqref{eq:ATE-inclusion}.

Thus \(e^{-tA}\) maps \(L^2(D)\) into \(C_\partial(\overline D)\).
The same holds for the bounded self-adjoint operator
\(T:=e^{-tA/3}\). Applying Theorem~1.1 of
\cite{ArendtElst2019} to the self-adjoint operator
$T$ on the bounded open set $D$, we deduce that
$T^3=e^{-tA}$ has an integral kernel in $C(\overline D\times\overline D)$,
namely $K_D(t,\cdot,\cdot)\in C(\overline D\times\overline D)$.

Fix $y\in D$ and $0<s<t$. By the semigroup property of the kernel,
\[
  K_D(t,x,y)=(e^{-(t-s)A}g)(x)
\]
for almost every $x\in D$ with
$g:=K_D(s,\cdot,y)$, and $g\in L^\infty(D)\subset L^2(D)$ by
\eqref{eq:gaussian-upper-bound}. By the mapping property just proved, the right-hand side lies in
$C_\partial(\overline D)$; since it agrees with $K_D(t,\cdot,y)$ almost
everywhere in $D$ and both functions are continuous on $\overline D$, we
conclude that $K_D(t,x,y)=0$ for $x\in\partial D$, $y\in D$. The boundary vanishing identity \eqref{eq:kernel-vanishing} then follows by the symmetry and continuity of the Dirichlet heat kernel.

(iv). Fix \(y\in D\) and \(t_0>0\), and choose \(0<s<t_0\).
By (iii), \(K_D(t,\cdot,y)=e^{-(t-s)A}g\) on \(\overline D\) for all
\(t>s\), where \(g:=K_D(s,\cdot,y)\in L^\infty(D)\). Continuity of
\(t\mapsto K_D(t,\cdot,y)\) at \(t_0\) is therefore equivalent to
continuity of \(\tau\mapsto e^{-\tau A}g\) at \(\tau_0:=t_0-s\). For
\(h>0\),
\[
e^{-(\tau_0+h)A}-e^{-\tau_0 A}
=e^{-\tau_0 A/2}\bigl(e^{-hA}-I\bigr)e^{-\tau_0 A/2},
\]
while for \(0<h<\tau_0\),
\[
e^{-\tau_0 A}-e^{-(\tau_0-h)A}
=e^{-(\tau_0-h)A/2}\bigl(e^{-hA}-I\bigr)e^{-(\tau_0-h)A/2}.
\]
By (ii), \(e^{-\tau A}:L^2(D)\to L^\infty(D)\) remains bounded for
\(\tau\) near \(\tau_0/2\). Combined with strong continuity of
\(e^{-tA}\) on \(L^2(D)\), both displays, applied to \(g\), tend to
zero in \(L^\infty(D)\) as \(h\to0^+\). Since
\(e^{-\tau A}g\in C_\partial(\overline D)\) by (iii), this is continuity
in \(C_\partial(\overline D)\). Evaluating at \(x\in\overline D\) yields
continuity of \((t,x)\mapsto K_D(t,x,y)\) on
\((0,\infty)\times\overline D\). The case of fixed \(x\) follows by
symmetry of the kernel.
\end{proof}

\begin{theorem}[Boundary vanishing of spectral projection diagonals] \label{thm:boundary-vanishing}
The diagonal value of the spectral projection kernel
    \[
    P_\lambda(y,y) = \sum_{j=1}^{m_\lambda} \phi_{\lambda,j}(y)^2
    \]
    is real-analytic in $D$, extends continuously to $\overline{D}$, and vanishes on $\partial D$. Moreover, every Dirichlet eigenfunction of $D$ with eigenvalue $\lambda$ extends continuously to $\overline{D}$ and vanishes on $\partial D$.
\end{theorem}
\begin{proof}
Each eigenfunction $\phi_{\lambda,j}$ is real-analytic in $D$ by interior analyticity for elliptic equations with analytic coefficients; see, for example, Morrey and Nirenberg \cite{MorreyNirenberg1957} or Morrey \cite{Morrey1966}. Since $P_\lambda(y,y)$ is a finite sum of squares of real-analytic functions, it is real-analytic in $D$.

It remains to show that $P_\lambda(y,y)$ is continuous on $\overline D$ and vanishes on $\partial D$. Fix an eigenfunction $\phi=\phi_{\lambda,j}\in H^1_0(D)$ satisfying $-\Delta\phi = \lambda\phi$ in $D$. It suffices to show that $\phi\in C(\overline D)$ and $\phi|_{\partial D}=0$, which we do via the Dirichlet heat semigroup $e^{-tA}$.

Since $\phi$ is an eigenfunction, $e^{-tA}\phi = e^{-\lambda t}\phi$ for every $t>0$. On the other hand, Lemma~\ref{lem:heat-kernel-boundary}(iii) shows that for each fixed $t>0$ the Dirichlet heat semigroup maps $L^2(D)$ into $C_\partial(\overline D)$, that is, $e^{-tA}f$ extends continuously to $\overline D$ and vanishes on $\partial D$ for every $f\in L^2(D)$. Hence
\[
e^{-tA}\phi \in C(\overline D)
\quad\text{and}\quad
(e^{-tA}\phi)|_{\partial D}=0.
\]
Multiplying by the nonzero scalar $e^{\lambda t}$ gives $\phi\in C(\overline D)$ and $\phi|_{\partial D}=0$.
The diagonal value $P_\lambda(y,y)=\sum_{j=1}^{m_\lambda}\phi_{\lambda,j}(y)^2$ inherits both the continuity on $\overline{D}$ and the pointwise vanishing on $\partial D$.
\end{proof}

The Gaussian bound \eqref{eq:gaussian-upper-bound} gives
\(K_D(t/2,y,\cdot)\in L^2(D)\) for fixed \(y\in D\) and
\(t>0\). Let \(\{(\lambda_n,\phi_n)\}_{n\ge1}\) be a real
orthonormal Dirichlet eigenbasis of \(D\), the eigenvalues being listed with
multiplicity. For every \(n\),
\[
\int_DK_D(t/2,y,z)\phi_n(z)\,dz=(e^{-tA/2}\phi_n)(y)
=e^{-t\lambda_n/2}\phi_n(y).
\]
This identity holds at \emph{every} \(y\in D\) because both sides are
continuous in \(y\), the left-hand side by
Lemma~\ref{lem:heat-kernel-boundary}(iii) and dominated convergence, and the
right-hand side because \(\phi_n\in C_\partial(\overline{D})\) by
Theorem~\ref{thm:boundary-vanishing}. Parseval's identity therefore gives
\[
\|K_D(t/2,y,\cdot)\|_{L^2(D)}^2
=\sum_{n\ge1}e^{-t\lambda_n}|\phi_n(y)|^2,
\]
the series converging absolutely. By the symmetry of the heat kernel and the
semigroup property, the left-hand side equals \(K_D(t,y,y)\):
\begin{equation} \label{eq:heat-diagonal-laplace0}
\|K_D(t/2,y,\cdot)\|_{L^2(D)}^2
=\int_DK_D(t/2,y,z)K_D(t/2,z,y)\,dz=K_D(t,y,y),
\end{equation}
where the first equality uses the symmetry of the kernel and the second the
semigroup property \(e^{-tA}=e^{-tA/2}e^{-tA/2}\). Grouping the terms of the
series by eigenvalue yields
\begin{equation} \label{eq:heat-diagonal-laplace}
K_D(t,y,y)=\sum_{\lambda\in\sigma(D)}e^{-t\lambda}P_\lambda(y,y)
=\int_0^\infty e^{-t\lambda}\,d\mu_y(\lambda),\quad t>0,
\end{equation}
the last identity being the definition \eqref{eq:spectral-measure} of \(\mu_y\).

\begin{lemma}[Strict positivity of the Dirichlet heat kernel]
\label{lem:heat-positivity}
For $t>0$, the Dirichlet heat kernel $K_D(t,\cdot, \cdot)$ is positive in $D\times D$, i.e.,
\[
K_D(t,x,y)>0,\quad t>0,\quad x,y\in D.
\]
\end{lemma}
\begin{proof}
Fix $y\in D$ and set $v(t,x):=K_D(t,x,y)$. First, $v\ge0$ on
$(0,\infty)\times D$ by Lemma~\ref{lem:heat-kernel-boundary}(i). Second,
by Lemma~\ref{lem:heat-kernel-boundary}(iii)--(iv), $v$ extends
continuously to $(0,\infty)\times\overline D$, and the extension vanishes
on the lateral surface $\partial D\times(0,\infty)$. Third, $v$ is smooth
on $(0,\infty)\times D$ and solves $(\partial_t-\Delta_x)v=0$ there: for
$0<\sigma<t$ one has $v(t,\cdot)=e^{-(t-\sigma)A}v(\sigma,\cdot)$ with
$v(\sigma,\cdot)\in L^\infty(D)\subset L^2(D)$, and iterating the
regularity argument of Lemma~\ref{lem:heat-kernel-boundary}(iii), based on
$\|A^me^{-sA}\|_{\mathcal L(L^2)}\le C_ms^{-m}$ with $C_m=(m/e)^m$, which follows from
$\sup_{\lambda>0}\lambda^me^{-s\lambda}=(m/(es))^m$, yields
$v(t,\cdot)\in\bigcap_{m\ge0}\mathcal{D}(A^m)$, whence $v\in C^\infty$ by interior
elliptic regularity and Sobolev embedding. Fourth, $v(t,\cdot)\not\equiv0$
for every $t>0$: by the symmetry of the kernel and the semigroup property,
as in \eqref{eq:heat-diagonal-laplace0},
\begin{equation}\label{eq:heat-kernel-nonvanishing}
\|v(t,\cdot)\|_{L^2(D)}^2=K_D(2t,y,y)
=\sum_{\lambda\in\sigma(D)}e^{-2t\lambda}P_\lambda(y,y)
\ge e^{-2t\lambda_1}\phi_1(y)^2>0,
\end{equation}
where $\lambda_1=\min\sigma(D)$ and $\phi_1$ is an $L^2(D)$-normalized
eigenfunction, which may be chosen positive because $D$ is connected
\cite[Section~6.5.1]{Evans2010}.

Suppose for contradiction that $v(t_0,x_0)=0$ for some $t_0>0$ and
$x_0\in D$. Fix $\tau$ with $0<\tau<t_0$ and consider the truncated
cylinder $Q_\tau:=\overline D\times[\tau,t_0]$. By the first two
properties, $v\in C(Q_\tau)$ and $v\ge0$ on $Q_\tau$; since $v(t_0,x_0)=0$,
the minimum of $v$ over $Q_\tau$ equals $0$ and is attained at
$(t_0,x_0)$, which is spatially interior and does not lie on the bottom
slice $t=\tau$. In other words, $(t_0,x_0)$ belongs to the parabolic
interior $D\times(\tau,t_0]$ of $Q_\tau$, the top time slice not being
part of the parabolic boundary. The strong maximum principle for the heat
equation \cite[Section~7.1.4]{Evans2010}, applied to $-v$ on this
cylinder, therefore implies that $v$ is constant on the backward slab
$D\times(\tau,t_0]$ and hence, by continuity, on
$\overline D\times[\tau,t_0]$; the constant equals the minimum
$v(t_0,x_0)=0$. Choosing any $t_1\in(\tau,t_0)$ we obtain $v(t_1,x)=0$
for all $x\in D$, that is, $\|v(t_1,\cdot)\|_{L^2(D)}=0$, which contradicts
\eqref{eq:heat-kernel-nonvanishing}. Since $y\in D$ was arbitrary and
$K_D$ is symmetric, the lemma follows: this is the pointwise form of the
positivity improving property of the Dirichlet heat semigroup.
\end{proof}

\subsection{Spectral Representation of the Green Function and the Backscattering Field}
Let $G_D(x,y,k)$ denote the Dirichlet Green's function, i.e., the distributional kernel of the resolvent $(A-k^2)^{-1}$:
\begin{equation}\label{eq:green-kernel-def}
[(A-k^2)^{-1}f](x)=\int_D G_D(x,y,k)f(y)\,dy,\quad k^2\notin\sigma(D).
\end{equation}
For fixed $x,y\in D$ satisfying $x\neq y$, the fundamental solution $\Phi_k(x,y)$ depends analytically on $k$ for $\Re k > 0$.
Therefore, all identities involving $\Phi_k$ below are understood on $\{\Re k>0\}$ away from the discrete set $k^2\in\sigma(D)$, while the meromorphic continuation of $G_D(x,y,\cdot)$ is supplied by the resolvent $(A-k^2)^{-1}$.

\begin{lemma}[Holomorphy of the kernels off the diagonal]
\label{lem:kernel-holomorphy}
Let $\Lambda\subset\mathbb C$ be a domain and let
$\mu\mapsto T(\mu)\in\mathcal L(L^2(D))$ be holomorphic on $\Lambda$ with
locally bounded operator norm. Let $K_T(\cdot,\cdot,\mu)\in\mathcal D'(D\times D)$
denote the distributional kernel of $T(\mu)$, and suppose that there is a
function $F\in C^\infty\bigl((D\times D)\setminus\Diag_{x=y}\bigr)$,
independent of $\mu$, such that
\begin{equation}\label{eq:kernel-holomorphy}
(-\Delta_x-\Delta_y-2\mu)K_T(\cdot,\cdot,\mu)=F
\quad\mbox{in }\mathcal D'\bigl((D\times D)\setminus\Diag_{x=y}\bigr)
\end{equation}
for every $\mu\in\Lambda$. Then $\mu\mapsto K_T(\cdot,\cdot,\mu)$ is
holomorphic with values in
$C^\infty_{\mathrm{loc}}\bigl((D\times D)\setminus\Diag_{x=y}\bigr)$. In
particular, for every fixed $x\neq y$ the map $\mu\mapsto K_T(x,y,\mu)$ is
holomorphic on $\Lambda$, and for every fixed $\mu\in\Lambda$ the function
$K_T(\cdot,\cdot,\mu)$ is smooth off the diagonal.
\end{lemma}
\begin{proof}
Fix separated relatively compact open sets $U_x,U_y\Subset D$ with
$\overline{U_x}\cap\overline{U_y}=\varnothing$, so that
$O_0:=U_x\times U_y$ misses the diagonal; it suffices to prove the claim
with $C^\infty_{\mathrm{loc}}(O_0)$ in place of
$C^\infty_{\mathrm{loc}}\bigl((D\times D)\setminus\Diag_{x=y}\bigr)$, since
$U_x,U_y$ are arbitrary. Throughout write $K(\cdot,\cdot,\mu):=K_T(\cdot,\cdot,\mu)$, let
$\Lambda'\Subset\Lambda$ be compact, and set
$M_{\Lambda'}:=\sup_{\mu\in\Lambda'}\|T(\mu)\|_{\mathcal L(L^2(D))}$,
which is finite since $\mu\mapsto T(\mu)$ is holomorphic.

\emph{Step~1: A $\mu$-uniform pairing bound.}
Let $\{(\lambda_n,\phi_n)\}_{n\ge1}$ be the Dirichlet eigenvalues of $A$,
repeated according to multiplicity, together with a corresponding real
orthonormal eigenbasis. For $\chi\in C_c^\infty(O_0)$, expand $\chi$ in
the eigenbasis $\{\phi_n\}$ acting in the $y$-variable,
\[
\chi(x,y)=\sum_na_n(x)\phi_n(y),
\quad a_n(x)=\int_D\chi(x,y)\phi_n(y)\,dy.
\]
Then
$\langle K(\cdot,\cdot,\mu),\chi\rangle
=\sum_n\langle T(\mu)\phi_n,a_n\rangle_{L^2(D)}$,
the series converging absolutely by the estimate below, and
$|\langle T(\mu)\phi_n,a_n\rangle|\le\|T(\mu)\|_{\mathcal L(L^2(D))}
\|a_n\|_{L^2(D)}$.
By Cauchy--Schwarz, together with $\sum_n(1+\lambda_n)^{-2}<\infty$, which
follows from Weyl's law $\lambda_n\sim 4\pi n/|D|$ as $n\to\infty$
\cite{Davies1989},
\[
\begin{aligned}
\sum_n\|a_n\|_{L^2(D)}
&\le\Bigl(\sum_n(1+\lambda_n)^{-2}\Bigr)^{1/2}
\Bigl(\sum_n(1+\lambda_n)^{2}\|a_n\|^2_{L^2(D)}\Bigr)^{1/2}\\
&=C\,\bigl\|(1+A_y)\chi\bigr\|_{L^2(D\times D)},
\end{aligned}
\]
where $A_y=-\Delta_y$ acts in the $y$-variable. Since
$\chi$ is compactly supported in $O_0\Subset D\times D$, the operator
identity $(1+A_y)\chi=(1-\Delta_y)\chi$ holds, and the last norm is
controlled by the ordinary Sobolev norm. Summarizing,
\begin{equation}\label{eq:kernel-compact-bound}
\|K(\cdot,\cdot,\mu)\|_{H^{-2}(O_0)}
\le C_{O_0}\,\|T(\mu)\|_{\mathcal L(L^2(D))}
\le C_{O_0}M_{\Lambda'},
\quad \mu\in\Lambda',
\end{equation}

\emph{Step~2: Weak holomorphy.}
If $\chi=\sum_j\alpha_j\otimes\beta_j$ is a finite sum with
$\alpha_j\in C_c^\infty(U_x)$ and $\beta_j\in C_c^\infty(U_y)$, then
\[
\langle K(\cdot,\cdot,\mu),\chi\rangle
=\sum_j\langle T(\mu)\beta_j,\alpha_j\rangle_{L^2(D)}
\]
is holomorphic in $\mu$. Such sums are dense in $C_c^\infty(O_0)$ in
the $H^2$ norm, so the bound \eqref{eq:kernel-compact-bound} shows that
$\mu\mapsto\langle K(\cdot,\cdot,\mu),\chi\rangle$, a locally uniform
limit of holomorphic functions, is holomorphic for every
$\chi\in C_c^\infty(O_0)$. Hence
$\mu\mapsto K(\cdot,\cdot,\mu)$ is weakly holomorphic in
$\mathcal D'(O_0)$.

\emph{Step~3: Elliptic improvement.}
Let $L_\mu:=-\Delta_x-\Delta_y-2\mu$. Its principal part is fixed and
$\mu$ enters only through the zeroth-order coefficient, while the
right-hand side $F$ of \eqref{eq:kernel-holomorphy} is smooth and
independent of $\mu$, so the interior elliptic estimate for $L_\mu$ holds
with constants that are uniform for $\mu\in\Lambda'$. Fix $O'\Subset O_0$ and
$\tau\ge0$, and choose nested open product domains
$O'=V_N\Subset V_{N-1}\Subset\cdots\Subset V_0=O_0$ with $2N>\tau+2$;
applying the interior estimate on $V_{k+1}\Subset V_k$ along this chain
gains two derivatives at each step, so the bound \eqref{eq:kernel-compact-bound}
improves to
\begin{equation}\label{eq:kernel-Htau-bound}
\|K(\cdot,\cdot,\mu)\|_{H^{\tau}(O')}\le C_{\tau,O'}(\Lambda')
\quad\text{for all }\mu\in\Lambda'.
\end{equation}

\emph{Step~4: From weak to norm holomorphy.}
It remains to pass from weak holomorphy to norm holomorphy in
$X:=H^{\tau}(O')$, a Hilbert space continuously and injectively embedded
in $\mathcal D'(O')$. We extend the scalar holomorphy from test functions
to all of $X^{\ast}$. For $\chi\in C_c^\infty(O')$, the pairing
$s_\chi:=\langle\,\cdot\,,\chi\rangle$ is a bounded linear functional on
$X$, with $\|s_\chi\|_{X^{\ast}}\le\|\chi\|_{H^{-\tau}(O')}$.
Let $M\subset X^{\ast}$ be the linear span of these pairings. Since
$M$ separates the points of $X$, its annihilator in $X$ is $\{0\}$,
so the weak-star closure of $M$ in $X^{\ast}$ is all of $X^{\ast}$ by
the Hahn--Banach theorem. Since $X$ is reflexive, the weak-star
topology $\sigma(X^{\ast},X)$ agrees with the weak topology
$\sigma(X^{\ast},X^{\ast\ast})$, and a linear subspace has the same
weak and norm closure. Hence $M$ is norm-dense in $X^{\ast}$.

Fix $\psi\in X^{\ast}$ and choose $\psi_n$ in this span with
$\psi_n\to\psi$ in $X^{\ast}$. Each function
$\mu\mapsto\langle K(\cdot,\cdot,\mu),\psi_n\rangle$ is holomorphic on
$\Lambda$, being a finite linear combination of the pairings above, and
\eqref{eq:kernel-Htau-bound} gives
\[
\bigl|\langle K(\cdot,\cdot,\mu),\psi-\psi_n\rangle\bigr|
\le C\,\|\psi-\psi_n\|_{X^{\ast}},
\quad \mu\in\Lambda',
\]
so $\mu\mapsto\langle K(\cdot,\cdot,\mu),\psi_n\rangle$ converges to
$\mu\mapsto\langle K(\cdot,\cdot,\mu),\psi\rangle$ locally uniformly on
$\Lambda$. Locally uniform limits of holomorphic functions are
holomorphic, so $\mu\mapsto K(\cdot,\cdot,\mu)$ is weakly holomorphic as
an $X$-valued map.

By the classical theorem of Dunford \cite{Dunford1938} (see also
\cite{HillePhillips1957}), a weakly holomorphic map from a domain in
$\mathbb C$ into a complex Banach space is holomorphic in the norm
topology. Hence $\mu\mapsto K(\cdot,\cdot,\mu)$ is holomorphic with values
in $H^\tau(O')$. For every $K\Subset O_0$ and every integer $m\ge0$, choose $O'$ with
$K\Subset O'\Subset O_0$ and $\tau>m+2$. Composing the
$H^\tau(O')$-holomorphy with the Sobolev embedding
$H^\tau(O')\hookrightarrow C^m(\overline K)$ shows that the map is
holomorphic with values in $C^m(\overline K)$. Since $K$ and $m$ are
arbitrary, the map is holomorphic with values in
$C^\infty_{\mathrm{loc}}(O_0)$, and hence all pointwise derivatives at
points of $O_0$ depend holomorphically on $\mu$. Varying the separated
sets $U_x,U_y$ yields holomorphy with values in
$C^\infty_{\mathrm{loc}}\bigl((D\times D)\setminus\Diag_{x=y}\bigr)$.
\end{proof}

\begin{theorem}[Meromorphic expansion of the resolvent kernel] \label{theo:spec-expansion}
As an operator-valued function, $(A-k^2)^{-1}$ has poles precisely at $k=\pm\sqrt{\lambda}$, $\lambda\in\sigma(D)$. For fixed $x\neq y$, the kernel $G_D(x,y,k)$ is meromorphic in $k$, with possible poles at these points, and for fixed $k$ with $k^2\notin\sigma(D)$ it is smooth off the diagonal, i.e., $G_D(\cdot,\cdot,k)\in C^\infty\bigl((D\times D)\setminus\Diag_{x=y}\bigr)$. The point $k=\sqrt{\lambda}$ is an actual pole if and only if $P_\lambda(x,y)\neq 0$, and similarly for $k=-\sqrt{\lambda}$. The residue at $k = \sqrt{\lambda}$ is
\begin{equation*}
\operatorname{Res}_{k = \sqrt{\lambda}} G_D(x, y, k) = -\frac{1}{2\sqrt{\lambda}}\, P_\lambda(x, y),
\end{equation*}
while the residue at $k=-\sqrt{\lambda}$ is
\begin{equation*}
\operatorname{Res}_{k = -\sqrt{\lambda}} G_D(x, y, k) = \frac{1}{2\sqrt{\lambda}}\, P_\lambda(x, y).
\end{equation*}
\end{theorem}
\begin{proof}
The proof is divided into three parts: the Laurent expansion of
$(A-\mu)^{-1}$ near an eigenvalue; the joint smoothness and holomorphy of
the kernels off the diagonal, which is deduced from
Lemma~\ref{lem:kernel-holomorphy}; and the poles and residues.

\emph{Part~1: Laurent expansion of $(A-\mu)^{-1}$.}
By the standard theory of the meromorphic resolvent \cite{Kato1995},
$\mu\mapsto(A-\mu)^{-1}$ is meromorphic on $\mathbb C$, with poles
precisely at the eigenvalues of $A$, which are isolated and of finite
multiplicity because $A$ has compact resolvent. Fix
$\lambda\in\sigma(D)$, let $P_\lambda$ be the orthogonal projection onto
$\ker(A-\lambda)$, set $Q:=I-P_\lambda$, and choose $\rho>0$ with
$\sigma(D)\cap B(\lambda,\rho)=\{\lambda\}$. The subspace
$QL^2(D)=\ker(A-\lambda)^\perp$ reduces $A$, and
$\sigma\bigl(A|_{QL^2(D)}\bigr)=\sigma(D)\setminus\{\lambda\}$. Hence
$A-\mu$ is boundedly invertible on $QL^2(D)$ for every
$\mu\in B(\lambda,\rho)$. Let $\iota:QL^2(D)\hookrightarrow L^2(D)$
denote the inclusion and set
\begin{equation}\label{eq:Hmu-def}
H(\mu):=\iota\circ\bigl((A-\mu)|_{QL^2(D)}\bigr)^{-1}\circ Q,
\quad \mu\in B(\lambda,\rho).
\end{equation}
As $\mu\mapsto((A-\mu)|_{QL^2(D)})^{-1}$ is holomorphic and $Q,\iota$
are fixed and bounded, $H$ is holomorphic on $B(\lambda,\rho)$ with
values in $\mathcal{L}(L^2(D))$.
Moreover, since $A$ commutes with $Q$, we have
\[
  H(\mu)=(A-\mu)^{-1}Q,\quad \mu\in B(\lambda,\rho)\setminus\{\lambda\}.
\]
On $\ker(A-\lambda)$ the resolvent equals $(\lambda-\mu)^{-1}I$, so on the
direct sum $L^2(D)=\ker(A-\lambda)\oplus QL^2(D)$ we have
\begin{equation}\label{eq:laurent-expansion}
(A-\mu)^{-1}=\frac{P_\lambda}{\lambda-\mu}+H(\mu),
\quad \mu\in B(\lambda,\rho)\setminus\{\lambda\}.
\end{equation}

\emph{Part~2: Joint smoothness and holomorphy of the kernels off the
diagonal.}
Fix separated relatively compact open sets $U_x,U_y\Subset D$ with
$\overline{U_x}\cap\overline{U_y}=\varnothing$, so that $U_x\times U_y$
misses the diagonal. Let $K_H(\cdot,\cdot,\mu)$ denote the
distributional kernel of $H(\mu)$ in $\mathcal D'(D\times D)$. From
\eqref{eq:Hmu-def} and the fact that $A$ commutes with $Q$ we have, for
every $\mu\in B(\lambda,\rho)$, the operator identities
$(A-\mu)H(\mu)=Q$ (on $L^2(D)$) and $H(\mu)(A-\mu)u=Qu$ for
$u\in\mathcal{D}(A)$.

Testing the first identity against
$\alpha(x)\beta(y)$ with $\alpha\in C_c^\infty(U_x)$,
$\beta\in C_c^\infty(U_y)$ and integrating by parts in $x$ gives
\[
(-\Delta_x-\mu)K_H=\delta_{\Diag_{x=y}}-P_\lambda(x,y)
\quad\text{in }\mathcal D'(U_x\times U_y),
\]
where $\delta_{\Diag_{x=y}}$ is the delta distribution carried by the
diagonal $\Diag_{x=y}$. Similarly, testing the second identity
$H(\mu)(A-\mu)u=Qu$ (with $u=\beta\in C_c^\infty(U_y)\subset
\mathcal{D}(A)$) against $\alpha(x)\beta(y)$ and integrating by parts
in $y$ gives
\[
(-\Delta_y-\mu)K_H=\delta_{\Diag_{x=y}}-P_\lambda(x,y)
\quad\text{in }\mathcal D'(U_x\times U_y).
\]
Symmetry is not invoked for the second identity, since $\mu$ need not
be real.
Restricting to the open product
$U_x\times U_y$, on which $\delta_{\Diag_{x=y}}=0$, and adding the two
identities yields the product elliptic equation
\begin{equation}\label{eq:product-equation}
    (-\Delta_x-\Delta_y-2\mu)K_H=-2P_\lambda(x,y)
\quad\text{in }\mathcal D'(U_x\times U_y).
\end{equation}
Its right-hand side is jointly smooth on $U_x\times U_y$, being the finite
sum $P_\lambda(x,y)=\sum_j\phi_{\lambda,j}(x)\phi_{\lambda,j}(y)$ of
products of eigenfunctions, each of which is real-analytic in $D$
by Theorem~\ref{thm:boundary-vanishing}.

Since $U_x,U_y\Subset D$ were arbitrary separated open sets,
\eqref{eq:product-equation} holds in fact in
$\mathcal D'\bigl((D\times D)\setminus\Diag_{x=y}\bigr)$. By Part~1, $\mu\mapsto H(\mu)$ is holomorphic
on $B(\lambda,\rho)$ with values in $\mathcal L(L^2(D))$, hence with
locally bounded operator norm, while the right-hand side
$-2P_\lambda$ of \eqref{eq:product-equation} is smooth on $D\times D$
and independent of $\mu$. Applying
Lemma~\ref{lem:kernel-holomorphy} with $T:=H$,
$\Lambda:=B(\lambda,\rho)$ and $F:=-2P_\lambda$, we conclude that
$\mu\mapsto K_H(\cdot,\cdot,\mu)$ is holomorphic with values in
$C^\infty_{\mathrm{loc}}\bigl((D\times D)\setminus\Diag_{x=y}\bigr)$; in
particular, for fixed $x\neq y$, the map $\mu\mapsto K_H(x,y,\mu)$ is
holomorphic on $B(\lambda,\rho)$.

The same reasoning applies to the resolvent itself. Let $\Gamma_\mu$
denote the distributional kernel of $(A-\mu)^{-1}$ for
$\mu\notin\sigma(D)$. By the same testing argument, applied to
$(A-\mu)(A-\mu)^{-1}=I$ and to $(A-\mu)^{-1}(A-\mu)u=u$ for
$u\in\mathcal D(A)$,
\[
(-\Delta_x-\mu)\Gamma_\mu=(-\Delta_y-\mu)\Gamma_\mu=\delta_{\Diag_{x=y}}
\quad\text{in }\mathcal D'(D\times D),
\]
hence $(-\Delta_x-\Delta_y-2\mu)\Gamma_\mu=0$ off the diagonal.
Since $\mu\mapsto(A-\mu)^{-1}$ is holomorphic
with locally bounded operator norm on the resolvent set
$\mathbb C\setminus\sigma(D)$, Lemma~\ref{lem:kernel-holomorphy} with
$T(\mu):=(A-\mu)^{-1}$, $\Lambda:=\mathbb C\setminus\sigma(D)$ and
$F:=0$ shows that $\mu\mapsto\Gamma_\mu$ is holomorphic with values in
$C^\infty_{\mathrm{loc}}\bigl((D\times D)\setminus\Diag_{x=y}\bigr)$.
In particular, for each fixed $k$ with $k^2\notin\sigma(D)$, by
\eqref{eq:green-kernel-def},
$G_D(x,y,k)=\Gamma_{k^2}(x,y)\in C^\infty\bigl((D\times D)\setminus\Diag_{x=y}\bigr)$,
and for fixed $x\neq y$ the map $k\mapsto G_D(x,y,k)$ is holomorphic
on $\{k:k^2\notin\sigma(D)\}$.

\emph{Part~3: Poles and residues.}
For fixed $x\neq y$, all terms of \eqref{eq:laurent-expansion} are smooth
off the diagonal and may be evaluated pointwise: for every $k$ with
$k^2\in B(\lambda,\rho)\setminus\{\lambda\}$, which entails
$k^2\notin\sigma(D)$,
\[
G_D(x,y,k)=\frac{P_\lambda(x,y)}{\lambda-k^2}+K_H(x,y,k^2),
\]
with $K_H(x,y,\cdot)$ holomorphic on $B(\lambda,\rho)$. As a function of
$k\in\mathbb C$ the right-hand side is meromorphic: its second term is
holomorphic on $\{k:k^2\in B(\lambda,\rho)\}$, while $\lambda-k^2$
vanishes simply at $k=\pm\sqrt{\lambda}$. Since the left-hand side is a
priori defined only on $\{k:k^2\notin\sigma(D)\}$, this identity supplies
the meromorphic continuation of $k\mapsto G_D(x,y,k)$ across the spectral
points. Combined with the holomorphy of $k\mapsto G_D(x,y,k)$ on
$\{k:k^2\notin\sigma(D)\}$ established above, this shows that
$k\mapsto G_D(x,y,k)$ is meromorphic on $\mathbb C$, with possible poles only at $k=\pm\sqrt{\lambda}$,
$\lambda\in\sigma(D)$, where the preceding
argument is applied at each such $\lambda$. Moreover, $k=\sqrt{\lambda}$
is an actual pole if and only if $P_\lambda(x,y)\neq0$, since
$K_H(x,y,k^2)$ is holomorphic near $k=\sqrt{\lambda}$.
Substituting $\mu=k^2$ and using
\[
\frac{1}{\lambda-k^2}
=\frac{1}{2\sqrt{\lambda}}\Bigl(\frac{1}{\sqrt{\lambda}-k}
+\frac{1}{\sqrt{\lambda}+k}\Bigr)
\]
gives
\[
\operatorname{Res}_{k=\sqrt{\lambda}}G_D(x,y,k)
=-\frac{1}{2\sqrt{\lambda}}P_\lambda(x,y),
\quad
\operatorname{Res}_{k=-\sqrt{\lambda}}G_D(x,y,k)
=+\frac{1}{2\sqrt{\lambda}}P_\lambda(x,y).
\]
By \eqref{eq:green-kernel-def} the kernel depends on $k$ only through
$k^2$, so $G_D(x,y,-k)=G_D(x,y,k)$; the same criterion therefore governs
the pole at $k=-\sqrt{\lambda}$.
\end{proof}

The pole criterion in Theorem~\ref{theo:spec-expansion} is an off-diagonal statement for fixed \(x\neq y\). The diagonal case is handled separately via the singularity cancellation in Lemma~\ref{lem:diagonal}. Since \(G_D(y,y,k)\) is not defined because of the logarithmic singularity at \(x=y\), the diagonal quantity is treated through \(u^s(y,y,k)\), where the free-space singularity cancels. Indeed, for $k\in\mathbb C$ with $\Re k>0$ and $k^2\notin\sigma(D)$,
\begin{equation}\label{eq:GD-decomposition}    
G_D(x,y,k)=\Phi_k(x,y)+u^s(x,y,k),\quad x\neq y, \, x, y\in D,
\end{equation}
because $G_D(\cdot,y,k)-\Phi_k(\cdot,y)$ solves the same Dirichlet problem \eqref{eq:helmholtz}--\eqref{eq:dirichlet} as the scattered field. 

\begin{lemma}[Spectral representation of the backscattering field] \label{lem:diagonal}
For each $y \in D$ and $k \in \mathbb{C}$ with $\Re k>0$ and $k^2 \notin \sigma(D)$, the backscattering field $u^s(y, y, k)$ is a real-analytic function of $y \in D$. As a function of $k$, $u^s(y, y, k)$ is meromorphic on the right half-plane $\{k \in \mathbb C : \Re k > 0\}$. Its poles are contained in $\{\sqrt{\lambda}:\lambda\in\sigma(D)\}$, and the residue at $k = \sqrt{\lambda}$ is
\begin{equation} \label{eq:residue-backscattering-rigorous}
\operatorname{Res}_{k = \sqrt{\lambda}} u^s(y, y, k) = -\frac{1}{2\sqrt{\lambda}}\, P_\lambda(y, y).
\end{equation}
\end{lemma}

\begin{proof}
The proof is divided into two steps: real analyticity in $y$, and meromorphy with the residue formula.

\emph{Step~1: Real analyticity in $y$.}
Fix $k$ with $\Re k>0$ and $k^2\notin\sigma(D)$. For each $w\in D$,
the solution $u^s(\cdot,w,k)\in H^1(D)$ of
\eqref{eq:helmholtz}--\eqref{eq:dirichlet} satisfies the homogeneous
equation \eqref{eq:helmholtz} in all of $D$ and is therefore
real-analytic in $D$ by interior elliptic regularity. Moreover,
standard continuous dependence for
\eqref{eq:helmholtz}--\eqref{eq:dirichlet}, together with the smoothness
of $(\xi,w)\mapsto\Phi_k(\xi,w)$ on $\partial D\times D$, gives that
$w\mapsto u^s(\cdot,w,k)$ is continuous from $D$ into $L^2(D)$; since
differences of solutions solve the homogeneous equation
\eqref{eq:helmholtz}, interior estimates make $u^s(\cdot,\cdot,k)$
jointly continuous on $D\times D$. In particular,
$u^s(\cdot,\cdot,k)\in L^1_{\mathrm{loc}}(D\times D)$ defines a
distribution on $D\times D$. Testing against $\chi\in C_c^\infty(D\times D)$ and using
\eqref{eq:helmholtz} together with Fubini's theorem gives
\[
(-\Delta_z-k^2)u^s=0\quad\text{in }\mathcal D'(D\times D).
\]
For the second variable we use the reciprocity relation
\[
u^s(z,w,k)=u^s(w,z,k),\quad z,w\in D.
\]
Indeed, Weyl's law $\lambda_n\sim 4\pi n/|D|$ gives
$\sum_{n\ge1}|\lambda_n-k^2|^{-2}<\infty$, so $(A-k^2)^{-1}$ is
Hilbert--Schmidt and its kernel admits the $L^2(D\times D)$-convergent
eigenfunction expansion
\[
G_D(z,w,k)=\sum_{n\ge1}\frac{\phi_n(z)\phi_n(w)}{\lambda_n-k^2}
=G_D(w,z,k)\quad\text{a.e. on }D\times D.
\]
Furthermore, by the off-diagonal continuity of the kernel
(Theorem~\ref{theo:spec-expansion}) and the symmetry of $\Phi_k$, the
decomposition \eqref{eq:GD-decomposition} yields the reciprocity
relation for $z\neq w$, and the diagonal case follows from the joint
continuity just established. Applying the previous paragraph with the
roles of $z$ and $w$ interchanged therefore gives
$(-\Delta_w-k^2)u^s=0$ in $\mathcal D'(D\times D)$, and adding the two
identities yields
\[
(-\Delta_z-\Delta_w-2k^2)u^s=0\quad\text{in }\mathcal D'(D\times D).
\]
This is an elliptic equation with constant, hence real-analytic,
coefficients, so analytic hypoellipticity (see, e.g., Morrey and
Nirenberg \cite{MorreyNirenberg1957}) shows that $u^s(\cdot,\cdot,k)$
coincides almost everywhere with a complex-valued real-analytic
function on $D\times D$; being jointly continuous,
$u^s(\cdot,\cdot,k)$ is itself real-analytic on $D\times D$. Finally,
for $z\neq y$ the decomposition \eqref{eq:GD-decomposition} gives
$u^s(z,y,k)=G_D(z,y,k)-\Phi_k(z,y)$, and letting $z\to y$ yields
\[
u^s(y,y,k)=\lim_{z\to y}\bigl(G_D(z,y,k)-\Phi_k(z,y)\bigr),
\]
so that $y\mapsto u^s(y,y,k)$ is real-analytic in $D$.

\emph{Step~2: Meromorphy and residue formula.}
Fix \(y_0\in D\) and \(r>0\) with \(B(y_0,2r)\subset D\), and let
\(\chi\in C_c^\infty(B(y_0,2r))\) satisfy \(\chi\equiv1\) on \(B(y_0,r)\).
Set
\[
v_\chi(\cdot,k):=G_D(\cdot,y_0,k)-\chi\,\Phi_k(\cdot,y_0).
\]
Since \(\chi\equiv1\) near \(y_0\), the singularity of \(G_D\) at \(y_0\)
is canceled by \(\chi\Phi_k\). Writing \([P,\chi]:=P\chi-\chi P\) for
the commutator of a differential operator \(P\) with multiplication by
\(\chi\), we obtain
\[
(-\Delta-k^2)v_\chi
=-\bigl[(-\Delta-k^2),\chi\bigr]\Phi_k(\cdot,y_0)=:f(\cdot,k),
\]
where \(f(\cdot,k)\) is holomorphic in \(k\) as a
\(C_c^\infty(D)\)-valued map, supported in
\(B(y_0,2r)\setminus B(y_0,r)\). We next verify that
\(v_\chi(\cdot,k)\in H^1_0(D)\). Since
\(G_D(\cdot,y_0,k)=\Phi_k(\cdot,y_0)+u^s(\cdot,y_0,k)\) away from \(y_0\),
the definition of \(v_\chi\) gives the decomposition
\[
v_\chi(\cdot,k)=(1-\chi)\,\Phi_k(\cdot,y_0)+u^s(\cdot,y_0,k).
\]
Because \(\chi\equiv1\) on \(B(y_0,r)\), the function \((1-\chi)\Phi_k\)
vanishes near \(y_0\) and is smooth on \(\overline D\); hence
\((1-\chi)\Phi_k(\cdot,y_0)\in H^1(D)\). Moreover
\(u^s(\cdot,y_0,k)\in H^1(D)\) solves \eqref{eq:helmholtz}--\eqref{eq:dirichlet},
so \(v_\chi(\cdot,k)\in H^1(D)\). Finally, \(\chi\) vanishes on a
neighborhood of \(\partial D\), while
\(\operatorname{Tr}u^s=-\operatorname{Tr}\Phi_k\) on \(\partial D\) by
\eqref{eq:dirichlet}; hence
\(\operatorname{Tr}v_\chi=\operatorname{Tr}\Phi_k-\operatorname{Tr}\Phi_k=0\),
and therefore $v_\chi(\cdot,k)\in H^1_0(D)$. Since
$f(\cdot,k)\in C_c^\infty(D)\subset L^2(D)$, the distributional
equation $(-\Delta-k^2)v_\chi=f(\cdot,k)$ yields
$-\Delta v_\chi=f(\cdot,k)+k^2v_\chi\in L^2(D)$, so
$v_\chi(\cdot,k)\in\mathcal D(A)$ and
$(A-k^2)v_\chi(\cdot,k)=f(\cdot,k)$ as an identity in $L^2(D)$.
Since $k^2\notin\sigma(D)$,
\begin{equation}\label{eq:v_chi-representation}
v_\chi(\cdot,k)=(A-k^2)^{-1}f(\cdot,k)
\quad\text{in }L^2(D).
\end{equation}
On $B(y_0,r)$ one has $\chi\equiv1$, so
$v_\chi(\cdot,k)=G_D(\cdot,y_0,k)-\Phi_k(\cdot,y_0)$ there, which by
\eqref{eq:GD-decomposition} equals $u^s(\cdot,y_0,k)$ on
$B(y_0,r)\setminus\{y_0\}$; both functions are continuous on
$B(y_0,r)$, the former by the construction of $v_\chi$ and the latter
by Step~1, and hence also $v_\chi(y_0,k)=u^s(y_0,y_0,k)$.

The right-hand side of \eqref{eq:v_chi-representation} is meromorphic as an \(L^2(D)\)-valued function of
\(k\) on \(\{\operatorname{Re}k>0\}\), with possible poles only at
\(k=\sqrt{\lambda}\), \(\lambda\in\sigma(D)\). Since \(f(\cdot,k)\) is
supported away from \(B(y_0,r)\), the same interior elliptic estimates
and Cauchy formula as in the proof of
Theorem~\ref{theo:spec-expansion} improve this to a meromorphic
\(C^\infty(B')\)-valued map for every \(B'\Subset B(y_0,r)\). In
particular, \(k\mapsto u^s(y_0,y_0,k)\) is meromorphic on the right
half-plane.

For the residue, \eqref{eq:laurent-expansion} gives, near
\(k=\sqrt{\lambda}\),
\[
v_\chi(\cdot,k)=\frac{P_\lambda f(\cdot,k)}{\lambda-k^2}+h(\cdot,k),
\]
with \(h\) holomorphic in \(H^m(B')\) for every \(m\). Since
\(f(\cdot,k)\) is holomorphic at \(k=\sqrt{\lambda}\) and
\(\operatorname{Res}_{k=\sqrt{\lambda}}(\lambda-k^2)^{-1}=-1/(2\sqrt{\lambda})\),
it remains to check
\((P_\lambda f(\cdot,\sqrt{\lambda}))(y_0)=P_\lambda(y_0,y_0)\).
The commutator \([(-\Delta-k^2),\chi]\Phi_k(\cdot,y_0)\) is smooth and
supported in \(D\setminus\{y_0\}\). For each eigenfunction
\(\phi_{\lambda,j}\),
\begin{align*}
\int_D\phi_{\lambda,j}\,f(\cdot,k)
&=-\bigl\langle\phi_{\lambda,j},\bigl[(-\Delta-k^2),\chi\bigr]
\Phi_k(\cdot,y_0)\bigr\rangle\\
&=-(\lambda-k^2)\bigl\langle\phi_{\lambda,j},\chi\Phi_k(\cdot,y_0)\bigr\rangle
+\phi_{\lambda,j}(y_0).
\end{align*}
At \(k=\sqrt{\lambda}\) this yields
\(\int_D\phi_{\lambda,j}f(\cdot,\sqrt{\lambda})=\phi_{\lambda,j}(y_0)\), hence
\[
\bigl(P_\lambda f(\cdot,\sqrt{\lambda})\bigr)(y_0)
=\sum_{j=1}^{m_\lambda}\phi_{\lambda,j}(y_0)^2=P_\lambda(y_0,y_0).
\]
As \(y_0\) was arbitrary, \eqref{eq:residue-backscattering-rigorous}
follows.
\end{proof}

The spectral measure extraction in Lemma~\ref{lem:diagonal} determines the diagonal value $P_\lambda(y,y)$ of the spectral projection kernel on the measurement curve $\M$. However, an eigenvalue $\lambda$ is detectable from the backscattering data only if its contribution does not vanish identically on $\M$. This motivates the following definition.

\begin{definition}[Visible and invisible eigenvalues] \label{def:visible}
For $\lambda \in \sigma(D)$, we say that $\lambda$ is \emph{visible} with respect to $\M$ if $P_\lambda(y,y)$ is not identically zero on $\M$; otherwise, $\lambda$ is called \emph{invisible} with respect to $\M$.
Equivalently, $\lambda$ is invisible with respect to $\M$ if and only if every eigenfunction in the $\lambda$-eigenspace vanishes on $\M$: choosing the eigenfunctions $\phi_{\lambda,j}$ to be real-valued, one has $P_\lambda(y,y)=\sum_{j=1}^{m_\lambda}\phi_{\lambda,j}(y)^2$, a sum of squares, so $P_\lambda(y,y)$ vanishes identically on $\M$ if and only if each $\phi_{\lambda,j}$ does. We say that an eigenspace is \emph{fully invisible} if all of its eigenfunctions vanish on $\M$, and \emph{partially visible} if it contains both eigenfunctions vanishing on $\M$ and eigenfunctions that do not vanish identically on $\M$.
Define $V(D):= \{\lambda \in \sigma(D): \lambda \text{ is visible with respect to } \M\}$.
\end{definition}

\begin{remark}[An invisible eigenspace on the disk]
\label{rmk:disk}
Invisibility is realized by concrete planar domains. Let $D = B(0,R)$ be a
disk and $\M = \{|x| = r_0\}$ a concentric
circle with $0 < r_0 < R$. The Dirichlet eigenfunctions are, up to
normalization constants,
\[
\phi_{m,n}(r,\theta) = J_m(j_{m,n} r/R)
\begin{cases}
\cos(m\theta) \\
\sin(m\theta)
\end{cases}, 
\quad
\lambda_{m,n}=(j_{m,n}/R)^2,
\]
where $j_{m,n}$ is the $n$-th positive zero of the Bessel
function $J_m$. Fix $m\ge0$ and $n\ge2$, choose $1\le\ell<n$, and set
$r_0=R\,j_{m,\ell}/j_{m,n}$, so that $0<r_0<R$. Then
$J_m(j_{m,n}r_0/R)=J_m(j_{m,\ell})=0$; hence the radial factor, and
therefore every eigenfunction of angular order $m$ and eigenvalue
$\lambda_{m,n}$, vanishes identically on $\M$.

For $\lambda_{m,n}$ to be \emph{fully} invisible, one must exclude
accidental coincidences with Bessel zeros of other orders. The full
eigenspace is
\[
E_{\lambda_{m,n}}(D)
=
\operatorname{span}\bigl\{\phi_{m',n'}:j_{m',n'}=j_{m,n}\bigr\},
\]
and if $j_{m',n'}=j_{m,n}$ for some $m'\ne m$, then $\phi_{m',n'}$ belongs
to the eigenspace but does not vanish on $\M$: its radial factor there is
$J_{m'}(j_{m,\ell})$, which is nonzero by Bourget's theorem recalled
below. In that case $\lambda_{m,n}$ is only partially visible. Bourget's
hypothesis, proved by Siegel \cite{Siegel1929}, rules such coincidences
out: for integers $p\ge0$ and $q\ge1$, the functions $J_p$ and $J_{p+q}$
have no common positive zeros. Applying this to
$\{m,m'\}=\{p,p+q\}$ gives $j_{m,n}\ne j_{m',n'}$ whenever $m'\ne m$.
Hence the eigenspace of $\lambda_{m,n}$ contains only angular modes of
order $m$ (one-dimensional for $m=0$ and two-dimensional for $m\ge1$) and
$\lambda_{m,n}$ is fully invisible from $\M$ for every $m\ge0$ and every
$n\ge2$.

Thus invisible eigenvalues do occur: there exist a domain $D$ and a
measurement curve $\M$ for which some Dirichlet eigenvalue of $D$ is fully
invisible from $\M$. Whether the multiplicity gap of
Theorem~\ref{thm:main-unique} is realized by a pair of \emph{distinct}
domains with identical curve data remains open; see
Section~\ref{sec:discussion}.
\end{remark}

\section{Main Results} \label{sec:uniqueness}

Throughout this section, let \(D_1,D_2\subset\mathbb R^2\) be bounded
connected Lipschitz domains, and let \(\Omega\Subset D_1\cap D_2\) be a
bounded connected Lipschitz domain whose boundary \(\M:=\partial\Omega\) is
connected. Let \(D^\ast\) be the connected component of \(D_1\cap D_2\) containing
\(\overline{\Omega}\).
See Figure~\ref{fig:setup} for an illustration of the geometric setup. 

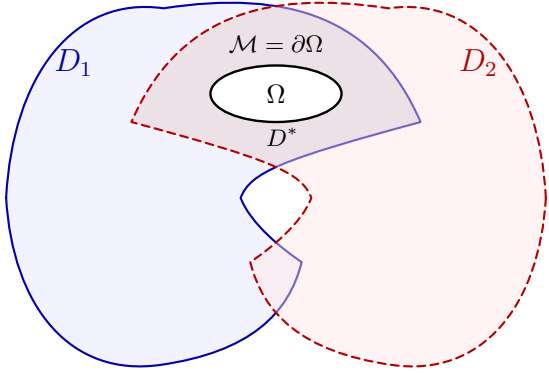
\begin{figure}[htbp]
\centering
\begin{tikzpicture}[scale=0.85, line join=round, line cap=round]
  \begin{scope}
    \clip
      (-4.2,0)
        .. controls (-4.1,1.9) and (-3.15,3.15) .. (-1.75,2.95)
        .. controls (0.2,3.25) and (1.55,2.8) .. (2.25,1.18)
        .. controls (0.1,0.6) and (-0.4,0.4) .. (-0.55,0.0)
        .. controls (-0.4,-0.4) and (0.1,-0.8) .. (0.4,-1.0)
        .. controls (0.2,-1.8) and (-0.4,-2.4) .. (-1.8,-2.6)
        .. controls (-3.2,-2.8) and (-4.1,-1.7) .. (-4.2,0) -- cycle;
    \clip
      (4.2,0)
        .. controls (4.1,1.9) and (3.15,3.15) .. (1.75,2.95)
        .. controls (-0.2,3.25) and (-1.55,2.8) .. (-2.25,1.18)
        .. controls (-0.1,0.6) and (0.4,0.4) .. (0.55,0.0)
        .. controls (0.4,-0.4) and (-0.1,-0.8) .. (-0.4,-1.0)
        .. controls (-0.2,-1.8) and (0.4,-2.4) .. (1.8,-2.6)
        .. controls (3.2,-2.8) and (4.1,-1.7) .. (4.2,0) -- cycle;
    \clip (-3.2,0.48) rectangle (3.2,3.4);
    \fill[gray!32] (-3.2,0.48) rectangle (3.2,3.4);
  \end{scope}
  \draw[thick, blue!65!black, fill=blue!10, fill opacity=0.45]
    (-4.2,0)
      .. controls (-4.1,1.9) and (-3.15,3.15) .. (-1.75,2.95)
      .. controls (0.2,3.25) and (1.55,2.8) .. (2.25,1.18)
      .. controls (0.1,0.6) and (-0.4,0.4) .. (-0.55,0.0)
      .. controls (-0.4,-0.4) and (0.1,-0.8) .. (0.4,-1.0)
      .. controls (0.2,-1.8) and (-0.4,-2.4) .. (-1.8,-2.6)
      .. controls (-3.2,-2.8) and (-4.1,-1.7) .. (-4.2,0);
  \draw[thick, red!70!black, densely dashed, fill=red!10, fill opacity=0.45]
    (4.2,0)
      .. controls (4.1,1.9) and (3.15,3.15) .. (1.75,2.95)
      .. controls (-0.2,3.25) and (-1.55,2.8) .. (-2.25,1.18)
      .. controls (-0.1,0.6) and (0.4,0.4) .. (0.55,0.0)
      .. controls (0.4,-0.4) and (-0.1,-0.8) .. (-0.4,-1.0)
      .. controls (-0.2,-1.8) and (0.4,-2.4) .. (1.8,-2.6)
      .. controls (3.2,-2.8) and (4.1,-1.7) .. (4.2,0);
  \draw[line width=0.9pt, fill=white]
    (0,1.62) ellipse (1.02 and 0.44);
  \node[blue!65!black, font=\large] at (-3.15,2.1) {$D_1$};
  \node[red!70!black, font=\large] at (3.15,2.1) {$D_2$};
  \node[font=\normalsize] at (0,1.62) {$\Omega$};
  \node[anchor=south, font=\small] at (0,2.12) {$\M=\partial \Omega$};
  \node[anchor=east, font=\small] at (0.5,0.95) {$D^\ast$};
\end{tikzpicture}
\caption{Geometric setup for Theorem~\ref{thm:main-unique}: $\M=\partial\Omega$ for a bounded Lipschitz domain $\Omega$,
and $D^\ast$ is the connected component of $D_1\cap D_2$ containing
$\overline{\Omega}$. The boundary of $D_1$ is drawn solid and that of $D_2$ dashed.}
\label{fig:setup}
\end{figure}

\subsection{Curve setting: spectral determination lemmas}
Throughout the comparison arguments below, equalities stated under the
condition $k^2 \notin \sigma(D_1)\cup \sigma(D_2)$ are understood in the
following sense: they hold for every frequency in the interval $(k_{\min},k_{\max})$ for
which both forward problems are well posed. Since the
spectra $\sigma(D_1)$ and $\sigma(D_2)$ are discrete, the exceptional set $\{k\in(k_{\min},k_{\max}):k^2\in\sigma(D_1)\cup\sigma(D_2)\}$
is finite,
though it depends on the unknown domains. Set
\begin{equation*}
  \mathcal{K}:=\{k\in(k_{\min},k_{\max}):k^2\notin\sigma(D_1)\cup\sigma(D_2)\}.
\end{equation*}

We first record the residue-comparison lemma for the curve-measurement setting.

\begin{lemma}[Visible spectrum from curve data] \label{lem:visible-spectrum}
If the multi-frequency backscattering data coincide on $\M$, i.e.,
\begin{equation*}
u^s_1(y, y, k) = u^s_2(y, y, k), \quad  y \in \M, \; k \in \mathcal{K},
\end{equation*}
then the following equalities hold.
\begin{enumerate}
\item[(1)] The visible Dirichlet eigenvalues coincide, i.e.,
      \begin{equation*}
          V(D_1) = V(D_2).
      \end{equation*}
\item[(2)] For every visible eigenvalue $\lambda \in V(D_1) = V(D_2)$, the spectral projection diagonals agree on $\M$, i.e.,
      \begin{equation*}
          P_\lambda^{(1)}(y,y) = P_\lambda^{(2)}(y,y),\quad y \in \M.
      \end{equation*}
\item[(3)] The spectral measures agree on $\M$:
      \begin{equation*}
          \mu_y^{(1)} = \mu_y^{(2)},\quad y \in \M.
      \end{equation*}
\end{enumerate}
\end{lemma}

\begin{proof}
Fix $y \in \M$. Lemma~\ref{lem:diagonal} shows that, for each $i=1,2$, the function $k\mapsto u^s_i(y,y,k)$ is meromorphic on the right half-plane $\{k \in \mathbb C : \Re k > 0\}$. Since the data coincide on $\mathcal{K}$, which has accumulation points in the right half-plane $\{k \in \mathbb C : \Re k > 0\}$, the identity theorem for meromorphic functions implies that $u^s_1(y,y,\cdot)$ and $u^s_2(y,y,\cdot)$ agree as meromorphic functions on $\{\Re k>0\}$. Hence the two Laurent expansions at $k=\sqrt{\lambda}$ coincide for every $\lambda\in\sigma(D_1)\cup\sigma(D_2)$. With the convention $P_\lambda^{(i)}\equiv0$ for $\lambda\notin\sigma(D_i)$, which renders the residue formula above valid for every $\lambda\in\sigma(D_1)\cup\sigma(D_2)$ (both sides being $0$ when $\lambda\notin\sigma(D_i)$), comparing the coefficients of $(k-\sqrt{\lambda})^{-1}$ therefore gives
\begin{equation} \label{eq:diagonal-equality}
P_\lambda^{(1)}(y,y) = P_\lambda^{(2)}(y,y)
\end{equation}
for every \(\lambda \in \sigma(D_1) \cup \sigma(D_2)\).

Let $\lambda$ be visible with respect to $\M$ for $D_1$. Then there exists $y_0 \in \M$ such that $P_\lambda^{(1)}(y_0,y_0) \neq 0$. Equation \eqref{eq:diagonal-equality} gives $P_\lambda^{(2)}(y_0,y_0) \neq 0$, so $\lambda$ is also visible for $D_2$. The reverse implication is identical. Therefore
$V(D_1) = V(D_2)$. For every visible eigenvalue $\lambda \in V(D_1)=V(D_2)$, equation \eqref{eq:diagonal-equality} holds for all $y \in \M$, proving (2).

Assertion (3) follows directly from \eqref{eq:diagonal-equality}. Indeed, for each $y \in \M$,
\begin{align*}
\mu_y^{(1)}
&= \sum_{\lambda \in \sigma(D_1)} P_\lambda^{(1)}(y,y)\,\delta_\lambda \\
&= \sum_{\lambda \in \sigma(D_1) \cup \sigma(D_2)} P_\lambda^{(1)}(y,y)\,\delta_\lambda \\
&= \sum_{\lambda \in \sigma(D_1) \cup \sigma(D_2)} P_\lambda^{(2)}(y,y)\,\delta_\lambda \\
&= \mu_y^{(2)}.
\end{align*}
\end{proof}

Before recording what the curve data determine, we isolate the pointwise
mechanism that underlies the Dirichlet heat-kernel.

\begin{lemma}[Pointwise determination of the heat-kernel diagonal]
\label{lem:point-heat-diagonal}
Let \(y_0\in D_1\cap D_2\). Assume that
\[
u^s_1(y_0, y_0, k) = u^s_2(y_0, y_0, k), \quad k \in \mathcal{K}.
\]
Then
\begin{equation}\label{eq:KD1=KD2}
K_{D_1}(t,y_0,y_0)=K_{D_2}(t,y_0,y_0), \quad t>0,
\end{equation}
where \(K_D\) denotes the Dirichlet heat kernel in \(D\).
\end{lemma}
\begin{proof}
Fix $y_0$. By Lemma~\ref{lem:diagonal}, for $i=1,2$ the function
$k\mapsto u^s_i(y_0,y_0,k)$ is meromorphic on the right half-plane
$\{k\in\mathbb C:\operatorname{Re}k>0\}$, and the data coincide on
$\mathcal K$, which has accumulation points in this half-plane. The
identity theorem for meromorphic functions therefore implies that
$u^s_1(y_0,y_0,\cdot)$ and $u^s_2(y_0,y_0,\cdot)$ agree as
meromorphic functions on $\{\operatorname{Re}k>0\}$. Comparing the
coefficients of $(k-\sqrt{\lambda})^{-1}$ in the two Laurent
expansions at $k=\sqrt{\lambda}$, with the convention
$P^{(i)}_\lambda\equiv0$ for $\lambda\notin\sigma(D_i)$, gives
\[
P^{(1)}_\lambda(y_0,y_0)=P^{(2)}_\lambda(y_0,y_0),
\quad \lambda\in\sigma(D_1)\cup\sigma(D_2),
\]
and consequently $\mu_{y_0}^{(1)}=\mu_{y_0}^{(2)}$, as in the proof
of Lemma~\ref{lem:visible-spectrum}(3). The equality \eqref{eq:KD1=KD2} then
follows by applying \eqref{eq:heat-diagonal-laplace} with $D=D_i$
and $y=y_0$, for $i=1,2$.
\end{proof}

The following strict monotonicity of the heat-kernel diagonal is also of independent interest.

\begin{lemma}[Strict monotonicity of the heat-kernel diagonal]
\label{lem:strict-monotonicity}
Let $D_1$ and $D_2$ be bounded connected Lipschitz domains such that $\overline{D_1}\subset D_2$. Then we have
\[
K_{D_1}(t,y,y)<K_{D_2}(t,y,y),\quad t>0,\,y\in D_1.
\]
\end{lemma}
\begin{proof}
The standard domain monotonicity of the Dirichlet heat kernel gives \cite{vandenBerg1989}
\[
K_{D_1}(t,x,y)\le K_{D_2}(t,x,y), \quad t>0, \quad x,y\in D_1.
\]
Fix $y\in D_1$ and set
\[
w(t,x):=K_{D_2}(t,x,y)-K_{D_1}(t,x,y), \quad t>0, \quad x\in D_1.
\]
Then $w\ge0$ on $(0,\infty)\times D_1$ and $(\partial_t-\Delta_x)w=0$
there. Since $\overline{D_1}\subset D_2$ and $D_2$ is open, we have
$\partial D_1\subset D_2$. Combining
Lemma~\ref{lem:heat-kernel-boundary}(iii) with
Lemma~\ref{lem:heat-positivity} yields
\[
K_{D_1}(t,x,y)=0<K_{D_2}(t,x,y), \quad t>0, \quad x\in\partial D_1,
\]
and therefore $w(t,x)>0$ on $\partial D_1$. Moreover, by Lemma~\ref{lem:heat-kernel-boundary}(iv) applied with
$D=D_1$ and with $D=D_2$, the function $w$ is continuous on
$(0,\infty)\times\overline{D_1}$. Following the proof of Lemma~\ref{lem:heat-positivity} yields $w(t,x)>0$ for all $t>0$ and all $x\in D_1$, and taking $x=y$ gives
the claim.
\end{proof}

From Lemmas~\ref{lem:visible-spectrum} and \ref{lem:point-heat-diagonal}, we deduce that the multi-frequency backscattering data on $\M$ determine
\begin{enumerate}
\item[(i)] the visible eigenvalue set $V(D)$;
\item[(ii)] the diagonal values $\{P_\lambda(y,y)\}_{\lambda\in V}$ of the
spectral projection kernels, and hence the pointwise spectral measures
$\mu_y$, $y\in\M$;
\item[(iii)] the diagonal of the Dirichlet heat kernel
$K_{D}(t,y,y)$ for all $t>0$, $y\in\M$.
\end{enumerate}
The present argument yields only the diagonal quantities
$P_\lambda(y,y)$ on $\M$, not the multiplicities or restriction ranks, which is the source of the conditional hypotheses in Theorem~\ref{thm:main-unique}.

By Theorem~\ref{thm:boundary-vanishing} applied to $D_1$ and $D_2$ separately, $P_\lambda^{(i)}(y,y)$ is real-analytic in $D_i$ (hence in $D^\ast \subset D_1 \cap D_2$) for $i=1,2$.
However, since $\M$ is a one-dimensional curve and $D^\ast$ is a two-dimensional open set, the identity theorem for real-analytic functions cannot extend the equality from $\M$ to $D^\ast$: the theorem requires vanishing on a nonempty open set, whereas we only know $P_\lambda^{(1)}(y,y) = P_\lambda^{(2)}(y,y)$ on the curve $\M$, and only for visible~$\lambda$. Moreover, Lemma~\ref{lem:visible-spectrum} only guarantees coincidence of visible eigenvalues; invisible eigenvalues may differ between $D_1$ and $D_2$.

To overcome this difficulty, we use the interior domain $\Omega$
enclosed by $\M$. In the remainder of this subsection, $D$ stands
for either $D_1$ or $D_2$.
We shall also need the Weyl law (see,
e.g., \cite{SafarovVassiliev1997})
\begin{equation} \label{eq:weyl-law}
N_D(E) = \frac{|D|}{4\pi} E + o(E) \quad \text{as } E \to \infty,
\end{equation} 
where $N_D(E)$ denotes the number of Dirichlet eigenvalues (counted with
multiplicity) not exceeding $E$. For each $\lambda\in\sigma(D)$, define the restriction operator
\[
R_\lambda:E_\lambda(D):=\ker(-\Delta_D-\lambda)\longrightarrow C(\M),
\quad
R_\lambda\phi:=\phi|_{\M}.
\]
Let
\[
N_{\mathrm{van},D}(E)
:=
\sum_{\substack{\lambda\in\sigma(D)\\ \lambda\le E}}
\dim\ker R_\lambda
\]
denote the total number of linearly independent eigenfunctions of $D$ with eigenvalue at most $E$ whose restrictions to $\M$ vanish.

\begin{lemma}[Restriction kernels and the interior Dirichlet spectrum] \label{lem:invisible-interior}
The restriction map $I_{\Omega}: \ker R_\lambda \longrightarrow \ker(-\Delta_\Omega-\lambda)$
is injective. In particular, $\dim\ker R_\lambda \le \dim\ker(-\Delta_\Omega-\lambda)$.
Moreover,
\begin{equation} \label{eq:vanishing-bound}
N_{\mathrm{van},D}(E)\le N_\Omega(E)
\quad\text{for all }E>0.
\end{equation}
\end{lemma}
\begin{proof}
Fix $\lambda\in\sigma(D)$, and let $\phi\in\ker R_\lambda$, so that $\phi|_{\M}=0$. Since $\phi$ satisfies
$(\Delta+\lambda)\phi=0$ in $D$, its restriction $\phi|_\Omega$ satisfies the same equation in $\Omega$ and vanishes on $\partial\Omega=\M$.
Because $\Omega$ is Lipschitz and $\phi\in C(\overline D)$ by Theorem~\ref{thm:boundary-vanishing}, the trace of $\phi|_\Omega$ on $\partial\Omega$ is zero. Hence
$\phi|_\Omega\in H^1_0(\Omega)$ solves $(\Delta+\lambda)\phi|_\Omega=0$ in $\Omega$. If $\phi|_\Omega=0$ in $\Omega$, then unique continuation for elliptic equations implies
$\phi\equiv0$ in $D$. Thus the restriction map $I_{\Omega}: \ker R_\lambda \longrightarrow \ker(-\Delta_\Omega-\lambda)$
is well defined and injective. Therefore $\dim\ker R_\lambda \le \dim\ker(-\Delta_\Omega-\lambda)$.
Summing over all $\lambda\le E$ yields \eqref{eq:vanishing-bound}.
\end{proof}

\subsection{Curve setting: spectral comparison and area estimate}

\begin{theorem}[Spectral comparison and area estimate]
\label{thm:main-unique}
Assume that
\[
u^s_1(y, y, k) = u^s_2(y, y, k),
\quad y \in \M, \quad k \in \mathcal{K}.
\] Then:

\begin{itemize}
\item[(a)] \textbf{Spectral-set comparison.} The symmetric difference
of the Dirichlet spectra satisfies
\[
\sigma(D_1)\triangle\sigma(D_2)\subset \sigma(\Omega).
\]
In particular, all eigenvalues of $D_1$ and $D_2$ outside
$\sigma(\Omega)$ coincide.

\item[(b1)] \textbf{Area estimate under multiplicity matching.}
If each common visible eigenvalue has the same total multiplicity in
$D_1$ and $D_2$, i.e.,
\[
m_\lambda(D_1)=m_\lambda(D_2)
\quad\text{for every }\lambda\in V(D_1)=V(D_2),
\]
then
\[
\bigl||D_1|-|D_2|\bigr|\le |\Omega|.
\]

\item[(b2)] \textbf{Area estimate under rank matching.}
If the restriction maps have equal ranks at each common visible
eigenvalue, i.e.,
\[
\operatorname{rank} R_\lambda^{(1)}=\operatorname{rank} R_\lambda^{(2)}
\quad\text{for every }\lambda\in V(D_1)=V(D_2),
\]
then the same area estimate holds:
\[
\bigl||D_1|-|D_2|\bigr|\le |\Omega|.
\]
\end{itemize}
\end{theorem}
\begin{proof}
We prove parts \textup{(a)}, \textup{(b1)}, and \textup{(b2)} separately.

\textbf{Proof of (a): Spectral set comparison.}
By Lemma~\ref{lem:visible-spectrum}, the visible eigenvalues of $D_1$ and $D_2$ coincide as sets: $V(D_1) = V(D_2) =: V$. By Lemma~\ref{lem:invisible-interior}, every invisible eigenvalue of $D_i$ is a Dirichlet eigenvalue of $\Omega$, i.e., $I(D_i) \subset \sigma(\Omega)$ for $i = 1, 2$, where $I(D_i) = \sigma(D_i) \setminus V$ is the set of invisible eigenvalues.

Since $\sigma(D_i) = V \cup I(D_i)$ (disjoint union), we have
\[
\sigma(D_1) \triangle \sigma(D_2) = (V \cup I_1) \triangle (V \cup I_2) \subset I_1 \cup I_2 \subset \sigma(\Omega),
\]
where $I_i = I(D_i)$. If $\lambda \in \sigma(D_1) \setminus \sigma(D_2)$, then $\lambda$ cannot be visible; otherwise $\lambda \in V(D_1) = V(D_2) \subset \sigma(D_2)$, a contradiction. Thus $\lambda \in I(D_1) \subset \sigma(\Omega)$. The case $\lambda \in \sigma(D_2) \setminus \sigma(D_1)$ is symmetric. This proves (a).

\textbf{Proof of (b1): Area estimate under multiplicity matching.}
For \(i=1,2\), write
\[
N_{D_i}(E)
= \sum_{\substack{\lambda\in V\\ \lambda\le E}} m_i(\lambda)
  + N_{\mathrm{inv},D_i}(E),
\]
where \(V=V(D_1)=V(D_2)\) by part~\textup{(a)}, \(m_i(\lambda)\) denotes the multiplicity of the visible eigenvalue \(\lambda\) in \(D_i\), and \(N_{\mathrm{inv},D_i}(E)\) counts the fully invisible eigenvalues of \(D_i\) up to energy \(E\) with multiplicity. By assumption, \(m_1(\lambda)=m_2(\lambda)\) for every \(\lambda\in V\). Hence
\[
N_{D_1}(E)-N_{D_2}(E)
=N_{\mathrm{inv},D_1}(E)-N_{\mathrm{inv},D_2}(E).
\]
Lemma~\ref{lem:invisible-interior} gives the stronger estimate
\[
N_{\mathrm{van},D_i}(E)\le N_\Omega(E), \quad i=1,2,
\]
where \(N_{\mathrm{van},D_i}(E)\) counts all eigenfunctions of \(D_i\) with eigenvalue at most \(E\) whose restrictions to \(\M\) vanish, including those belonging to partially visible eigenspaces. Since every fully invisible eigenvalue contributes its full multiplicity to \(N_{\mathrm{van},D_i}(E)\), we have
\[
0\le N_{\mathrm{inv},D_i}(E)\le N_{\mathrm{van},D_i}(E)\le N_\Omega(E), \quad i=1,2,
\]
and therefore
\[
|N_{D_1}(E)-N_{D_2}(E)|\le N_\Omega(E).
\]
Dividing by \(E\) and applying the Weyl law \eqref{eq:weyl-law} to both sides yields
\[
\frac{\bigl||D_1| - |D_2|\bigr|}{4\pi} + o(1) \leq \frac{|\Omega|}{4\pi} + o(1) \quad \text{as } E \to \infty.
\]
Taking $E \to \infty$ yields
\[
\bigl||D_1| - |D_2|\bigr| \leq |\Omega|,
\]
which completes the proof of (b1).

\textbf{Proof of \textup{(b2)}: Area estimate under rank matching.}
For each eigenvalue \(\lambda\) we have
\[
m_i(\lambda) = \operatorname{rank}R_\lambda^{(i)} + \dim\ker R_\lambda^{(i)},
\]
and \(\operatorname{rank}R_\lambda^{(i)}=0\) for \(\lambda\notin V\).
Summing over \(\lambda\le E\) and separating the visible eigenvalues
from the vanishing directions gives the decomposition
\[
N_{D_i}(E)
=
\sum_{\substack{\lambda\in V\\ \lambda\le E}}
\operatorname{rank} R_\lambda^{(i)}
+
N_{\mathrm{van},D_i}(E),
\]
where the first sum counts the visible directions (the rank of the
restriction map) and the second term counts all vanishing directions,
including both fully invisible eigenvalues and the kernel directions
of partially visible eigenspaces.

Under the rank-matching assumption,
\[
\operatorname{rank}R_\lambda^{(1)}=\operatorname{rank}R_\lambda^{(2)}
\quad\text{for every }\lambda\in V(D_1)=V(D_2),
\]
the visible-rank sums cancel when taking the difference:
\[
N_{D_1}(E)-N_{D_2}(E)
=
N_{\mathrm{van},D_1}(E)-N_{\mathrm{van},D_2}(E).
\]
By Lemma~\ref{lem:invisible-interior},
\(0\le N_{\mathrm{van},D_i}(E)\le N_\Omega(E)\), so
\[
|N_{D_1}(E)-N_{D_2}(E)|
\le N_\Omega(E).
\]
The proof is completed by using Weyl asymptotics, exactly as in \textup{(b1)}.
\end{proof}

\begin{remark}[Multiplicity gap and sufficient conditions] \label{rmk:simplicity}
Theorem~\ref{thm:main-unique} is asymmetric: part~(a) is unconditional, while the area estimates in (b1) and (b2) require additional spectral
hypotheses beyond what the diagonal data determine. The two conditions, multiplicity matching and restriction-rank matching, are
logically independent: neither implies the other in general. A common
sufficient special case is when every visible eigenvalue is simple in both
domains: if $m_\lambda(D_i)=1$ for all $\lambda\in V$, then both (b1) and
(b2) hold automatically.
\end{remark}

\subsection{Open-set setting}
In the open-set setting, we fix the wavenumber $k_0>0$ such that $k_0^2\notin\sigma(D)$ and show uniqueness of the inverse problem from
the backscattering data in a nonempty open set $\mathcal U$. 

An important tool is a quantitative nondegeneracy property
of harmonic measure near the boundary
(Lemma~\ref{lem:harmonic-measure-nondegenerate}), which we deduce from
the theory of nontangentially accessible (NTA) domains in the sense of
Jerison and Kenig \cite{JerisonKenig1982}. A bounded Lipschitz domain
is an NTA domain \cite[Section~3]{JerisonKenig1982}; in particular, it
enjoys the defining properties \textup{(NTA1)} and \textup{(NTA2)}
below. We also use the consequence \textup{(NTA3)}. All constants depend only on
the Lipschitz character of $D$.
\begin{enumerate}
\item[(NTA1)] \emph{Corkscrew condition}: there exist $a\in(0,1)$,
$M_1>1$ and $\delta_0>0$ such that for every $p\in\partial D$
and every $0<\rho<\delta_0$ there is a point
$A_\rho(p)\in D\cap B(p,M_1\rho)$ with
$\operatorname{dist}(A_\rho(p),\partial D)\ge a\rho$.
\item[(NTA2)] \emph{Bounded-length Harnack chains}: if $x_1,x_2\in D$
satisfy $\operatorname{dist}(x_i,\partial D)\ge\alpha|x_1-x_2|$ for
some $\alpha>0$, then $x_1$ and $x_2$ can be joined by a Harnack
chain of balls of length bounded in terms of $\alpha$ and the NTA
character of $D$. 
\item[(NTA3)] \emph{Surface-ball nondegeneracy of harmonic measure}:
there are constants $M_2>1$, $c>0$, and $\delta_1>0$ such that
for every $p\in\partial D$, every $0<\rho<\delta_1$, and any
interior corkscrew point $A_\rho(p)$ as in \textup{(NTA1)},
\[
\omega^{A_\rho(p)}\bigl(\Sigma(p,M_2\rho)\bigr)\ge c,
\qquad
\Sigma(p,r):=\partial D\cap B(p,r),
\]
where $\omega^z$ denotes harmonic measure for $D$ with pole at $z$.
\end{enumerate}

\begin{lemma}[Nondegeneracy of harmonic measure]
\label{lem:harmonic-measure-nondegenerate}
There exist constants $M>1$, $c_0>0$, and $\eta>0$, depending only on
the Lipschitz character of $D$, such that for every $y\in D$ with
\[
\delta(y):=\operatorname{dist}(y,\partial D)<\eta,
\]
and any $p\in\partial D$ with $|y-p|=\delta(y)$, we have
\[
\omega^y(\Sigma(p,M\delta(y)))\ge c_0.
\]
\end{lemma}
\begin{proof}
Set $M:=\max\{M_1,M_2\}>1$ and $\eta:=\min\{\delta_0,\delta_1\}$, so that
\textup{(NTA1)}--\textup{(NTA3)} hold with this common $M$ and with every
scale $\rho<\eta$. Set $r:=\delta(y)<\eta$ and choose $p\in\partial D$
with $|y-p|=r$. By
\textup{(NTA1)} there is a corkscrew point
$A_r(p)\in D\cap B(p,Mr)$ with
$\operatorname{dist}(A_r(p),\partial D)\ge ar$. Put
$d:=|y-A_r(p)|$; then $d\le(M+1)r$, so
\[
\operatorname{dist}(y,\partial D)=r\ge\frac{d}{M+1},
\qquad
\operatorname{dist}(A_r(p),\partial D)\ge ar\ge\frac{a}{M+1}\,d
\]
Applying \textup{(NTA2)} with $\alpha=\frac{a}{M+1}$, with the help of the interior Harnack inequality, we derive that
\[
Cu(A_r(p))\ge u(y)\ge C^{-1}u(A_r(p))
\]
for every positive harmonic function $u$
in $D$, with $C\ge1$ depending only on the Lipschitz character of
$D$. The function $u(z):=\omega^z(\Sigma(p,Mr))$ is positive and
harmonic in $D$ (see, e.g., \cite[Chapter~1]{Kenig1994}); since
$M\ge M_2$ we have $\Sigma(p,M_2r)\subset\Sigma(p,Mr)$, and
\textup{(NTA3)} gives
\[
u(A_r(p))\ge\omega^{A_r(p)}(\Sigma(p,M_2r))\ge c.
\]
Hence
\[
\omega^y(\Sigma(p,Mr))=u(y)\ge C^{-1}u(A_r(p))\ge C^{-1}c=:c_0,
\]
with $c_0>0$ depending only on the Lipschitz character of $D$.
\end{proof}

\begin{lemma}[Boundary behavior of the harmonic Robin function]
\label{lem:harmonic-robin-boundary}
Let $G_D(\cdot, y, 0)$ denote the Dirichlet Green function of $-\Delta$ in $D$. Define
\[
R_D^{(0)}(y):=\lim_{x\to y}\Bigl(G_D(x,y,0)+\frac{1}{2\pi}\log|x-y|\Bigr),
\quad y\in D.
\]
Then
\[
R_D^{(0)}(y)\to -\infty
\quad\text{as } y\to\partial D,\ y\in D.
\]
\end{lemma}
\begin{proof}
Fix $y\in D$. For $x\neq y$, set
\[
H_y(x):=G_D(x,y,0)+\frac{1}{2\pi}\log|x-y|.
\]
The logarithmic singularities cancel, so $H_y$ extends harmonically across $x=y$, and
\[
R_D^{(0)}(y)=H_y(y).
\]
For fixed $y\in D$, the boundary function
$\xi\longmapsto \frac{1}{2\pi}\log|\xi-y|$
is continuous on \(\partial D\), and the Dirichlet Green function \(G_D(\cdot,y,0)\) vanishes continuously on \(\partial D\). Hence the boundary trace of \(H_y\) is
\[
H_y(\xi)=\frac{1}{2\pi}\log|\xi-y|,
\quad \xi\in\partial D.
\]
Therefore, the harmonic extension is represented by harmonic measure, and
\begin{equation}\label{R_D^0-representation}
    R_D^{(0)}(y)=\int_{\partial D}\frac{1}{2\pi}\log|\xi-y|\,d\omega^y(\xi).
\end{equation}

Let $\delta(y)=\operatorname{dist}(y,\partial D)$, and choose $p\in\partial D$ such that $|y-p|=\delta(y)$. Let $M,c_0,\eta$ be as in Lemma~\ref{lem:harmonic-measure-nondegenerate}. Then we have
\[
|\xi-y|\le (M+1)\delta(y),
\]
and therefore
\[
\frac{1}{2\pi}\log|\xi-y|
\le \frac{1}{2\pi}\log((M+1)\delta(y))\quad\text{for all }\xi\in\Sigma(p,M\delta).
\]
On the other hand,
\[
|\xi-y|\le \operatorname{diam}(D)\quad\text{for all }\xi\in\partial D.
\]
Splitting the integral \eqref{R_D^0-representation} over $\Sigma(p,M\delta)$ and its complement, we obtain
\begin{align*}
R_D^{(0)}(y)
&\le
\omega^y(\Sigma(p,M\delta))
\frac{1}{2\pi}\log((M+1)\delta(y)) \\
&\quad+
\bigl(1-\omega^y(\Sigma(p,M\delta))\bigr)
\frac{1}{2\pi}\log\operatorname{diam}(D).
\end{align*}
By Lemma~\ref{lem:harmonic-measure-nondegenerate},
\[
\omega^y(\Sigma(p,M\delta))\ge c_0
\]
for all $y\in D$ with
$\delta(y)<\min\{\eta,\operatorname{diam}(D)/(M+1)\}$. Since
\[
\frac{1}{2\pi}\log((M+1)\delta(y))
\le
\frac{1}{2\pi}\log\operatorname{diam}(D)
\]
for all such $y$, we may replace $\omega^y(\Sigma(p,M\delta))$ by the lower bound $c_0$ in the preceding estimate and obtain
\[
R_D^{(0)}(y)\le \frac{1}{2\pi}\Bigl(c_0\log((M+1)\delta(y))+(1-c_0)\log\operatorname{diam}(D)\Bigr),
\]
and the conclusion follows since $\delta(y)\to0$ implies $\log\delta(y)\to-\infty$.
\end{proof}

\begin{lemma}[Comparison of diagonal regular parts]
\label{lem:diag-regular-part-comparison}
Let $G_D(\cdot, y, k_0)$ be the Dirichlet Green function of $-\Delta-k_0^2$ in $D$. Define
\[
R_D^{(k_0)}(y):=\lim_{x\to y}\bigl(G_D(x,y,k_0)-\Phi_{k_0}(x,y)\bigr),
\quad y\in D.
\]
Then $R_D^{(k_0)}(y)-R_D^{(0)}(y)$ remains bounded as $y\to\partial D$, $y\in D$.
\end{lemma}
\begin{proof}
Write $G_0(x,y):=G_D(x,y,0)$ and $g_y:=G_0(\cdot,y)$, and set
$\mathcal R_{k_0}:=(A-k_0^2)^{-1}$, $\mathcal R_0:=A^{-1}$.

\emph{Step~1: A two-sided bound for $G_0$ and its consequences.}
The positivity-preserving property of $A^{-1}$ gives $G_0\ge0$ almost
everywhere on $D\times D$. Since $G_0$ is smooth off the diagonal by
Theorem~\ref{theo:spec-expansion}, the bound holds pointwise for
$x\neq y$. For the upper bound, set
\[
u_y(x):=\frac{1}{2\pi}\log\frac{\operatorname{diam}(D)}{|x-y|}-G_0(x,y),
\quad x\in \overline{D}\setminus\{y\}.
\]
The logarithmic singularities cancel at $x=y$, so $u_y$ extends
harmonically to $D$. Every boundary point of the Lipschitz domain $D$
is Wiener regular, hence $G_0(\cdot,y)$ vanishes continuously on
$\partial D$. Since $|x-y|\le\operatorname{diam}(D)$ for $x\in\partial D$, $u_y\ge0$
on $\partial D$, and the maximum principle gives $u_y\ge0$ in $D$.
Combined with the lower bound,
\begin{equation}\label{eq:G0-two-sided}
0\le G_0(x,y)\le\frac{1}{2\pi}\log\frac{\operatorname{diam}(D)}{|x-y|},
\quad x,y\in D,\ x\neq y.
\end{equation}
In particular $g_y\in L^2(D)$ with
$\sup_{y\in D}\|g_y\|_{L^2(D)}<\infty$, and $y\mapsto g_y$ is
continuous from $D$ into $L^2(D)$: for $y'$ close to $y$ and small
$\rho>0$, the integral of $|g_{y'}-g_y|^2$ over $B(y,2\rho)$ is
$O\bigl(\rho^2(1+|\log\rho|^2)\bigr)$ uniformly in $y'$ by
\eqref{eq:G0-two-sided}, while over $D\setminus B(y,2\rho)$ it tends
to zero by dominated convergence, using the pointwise convergence
$g_{y'}(x)\to g_y(x)$ and the $y'$-uniform majorant
$\frac1\pi\log\bigl(2\operatorname{diam}(D)/\rho\bigr)\in L^2(D)$.

\emph{Step~2: $\mathcal R_{k_0}:L^2(D)\to C_\partial(\overline D)$ is
bounded.} By Lemma~\ref{lem:heat-kernel-boundary}(ii)--(iii), $e^{-tA}$
maps $L^2(D)$ into $C_\partial(\overline D)$ with
$\|e^{-tA}f\|_{L^\infty(D)}\le C(t)\|f\|_{L^2(D)}$, where
$C\in L^1(0,\infty)$ and
$\int_0^\infty C(t)\,dt=2C_1+C_2e^{-\lambda_1}/\lambda_1$. Moreover
$t\mapsto e^{-tA}f$ is continuous from $(0,\infty)$ into
$C_\partial(\overline D)$, as one sees by writing
$e^{-tA}f=e^{-sA}\bigl(e^{-(t-s)A}f\bigr)$ with $0<s<t$. Hence
$\int_0^\infty e^{-tA}f\,dt$ converges as a Bochner integral in
$C_\partial(\overline D)$. By the spectral theorem it equals $A^{-1}f$
in $L^2(D)$, and the two limits coincide since
$C_\partial(\overline D)\hookrightarrow L^2(D)$. This gives
\begin{equation}\label{eq:inverse-laplacian-bound}
\|A^{-1}f\|_{L^\infty(D)}
\le\Bigl(2C_1+\frac{C_2e^{-\lambda_1}}{\lambda_1}\Bigr)\|f\|_{L^2(D)},
\quad f\in L^2(D).
\end{equation}
Since $k_0^2\notin\sigma(D)$, the compact self-adjoint operator
$k_0^2A^{-1}$ does not have $1$ in its spectrum; thus $I-k_0^2A^{-1}$
is invertible on $L^2(D)$, and the factorization
$\mathcal R_{k_0}=A^{-1}(I-k_0^2A^{-1})^{-1}$ shows that
$\mathcal R_{k_0}:L^2(D)\to C_\partial(\overline D)$ is bounded.
Denote its norm by $C_{\mathrm{res}}$.

\emph{Step~3: The difference $w_y:=G_D(\cdot,y,k_0)-g_y$ as a
resolvent image.} The singularities of the two Green functions at
$x=y$ cancel: by \eqref{eq:GD-decomposition},
$G_D(\cdot,y,k_0)=\Phi_{k_0}(\cdot,y)+u^s(\cdot,y,k_0)$ with
$u^s(\cdot,y,k_0)\in C^\infty(D)$, while near $y$,
$g_y=\Phi_0(\cdot,y)+H_y$ with $\Phi_0(x,y):=-\frac1{2\pi}\log|x-y|$
and $H_y$ harmonic; the small-argument expansion of the Hankel
function,
\begin{equation}\label{eq:hankel-expansion}
\Phi_{k_0}(x,y)-\Phi_0(x,y)=c_{k_0}+O\bigl(|x-y|^2\log|x-y|\bigr)
\quad\text{as }x\to y,
\end{equation}
with
\[
c_{k_0}:=-\frac{1}{2\pi}\Bigl(\log\frac{k_0}{2}+\gamma\Bigr)+\frac{i}{4},
\]
then shows that $w_y$ extends continuously across $x=y$.
(Subtracting the distributional identities
$(-\Delta_x-k_0^2)G_D(\cdot,y,k_0)=\delta_y$ and
$-\Delta_xg_y=\delta_y$ formally gives
$(-\Delta_x-k_0^2)w_y=k_0^2g_y$; the representation we now establish
is the rigorous kernel-level form of this identity.)

Put $F(x,y):=k_0^2(\mathcal R_{k_0}g_y)(x)$ for $x\in\overline D$,
$y\in D$. By Steps~1--2,
\begin{equation}\label{eq:F-uniform-bound1}
\|F(\cdot,y)\|_{L^\infty(D)}
\le k_0^2C_{\mathrm{res}}\sup_{y\in D}\|g_y\|_{L^2(D)}<\infty,
\end{equation}
and
\begin{equation}\label{eq:F-uniform-bound2}
\|F(\cdot,y)-F(\cdot,y')\|_{L^\infty(D)}
\le k_0^2C_{\mathrm{res}}\|g_y-g_{y'}\|_{L^2(D)}\to0\ (y'\to y),
\end{equation}
so $F\in C(\overline D\times D)$ by the triangle inequality. For
$f,\varphi\in C_c^\infty(D)$, Fubini's theorem and the second
resolvent identity
$\mathcal R_{k_0}-\mathcal R_0=k_0^2\mathcal R_{k_0}\mathcal R_0$ give
$$
\iint_{D\times D}F(x,y)f(y)\varphi(x)\,dy\,dx
=k_0^2\langle\mathcal R_{k_0}\mathcal R_0f,\varphi\rangle
=\langle(\mathcal R_{k_0}-\mathcal R_0)f,\varphi\rangle.
$$
The locally integrable kernel $G_D(\cdot,\cdot,k_0)-G_0$ of
$\mathcal R_{k_0}-\mathcal R_0$ satisfies the same pairing identity,
and finite sums $f\otimes\varphi$ are dense in
$C_c^\infty(D\times D)$; hence the two functions agree almost
everywhere on $D\times D$. Both are continuous off the diagonal ($F$ on all of $\overline D\times D$, the kernel by
Theorem~\ref{theo:spec-expansion}) so
\begin{equation}\label{eq:wy-representation}
w_y(x)=F(x,y)=k_0^2(\mathcal R_{k_0}g_y)(x),
\quad x,y\in D,\ x\neq y.
\end{equation}

\emph{Step~4: Evaluation on the diagonal.} For fixed $y\in D$, both
sides of \eqref{eq:wy-representation} are continuous at $x=y$ (the
left-hand side by Step~3, the right-hand side since
$F\in C(\overline D\times D)$). Letting $x\to y$ gives
$w_y(y)=F(y,y)$. Using \eqref{eq:hankel-expansion},
$$
R_D^{(k_0)}(y)-R_D^{(0)}(y)=w_y(y)-c_{k_0}=F(y,y)-c_{k_0},
$$
and \eqref{eq:F-uniform-bound1}-\eqref{eq:F-uniform-bound2} yields
$$
\sup_{y\in D}\bigl|R_D^{(k_0)}(y)-R_D^{(0)}(y)\bigr|
\le k_0^2C_{\mathrm{res}}\sup_{y\in D}\|g_y\|_{L^2(D)}+|c_{k_0}|<\infty,
$$
which proves the lemma.
\end{proof}

\begin{lemma}[Boundary blow-up of the diagonal regular part]
\label{lem:diag-regular-part-blowup}
Let $k_0>0$ satisfy $k_0^2\notin \sigma(D)$.
Then
\[
\Re\,R_D^{(k_0)}(y)\to -\infty
\quad\text{as } y\to\partial D,\ y\in D.
\]
Equivalently,
\[
\Re\,u^s(y,y,k_0)\to -\infty
\quad\text{as } y\to\partial D,\ y\in D.
\]
\end{lemma}

\begin{proof}
By Lemma~\ref{lem:diagonal}, the diagonal value $u^s(y,y,k_0)$ is well defined and satisfies
\[
R_D^{(k_0)}(y)=u^s(y,y,k_0),\quad y\in D.
\]
Moreover, Lemma~\ref{lem:diag-regular-part-comparison} shows that
$R_D^{(k_0)}(y)-R_D^{(0)}(y)$ remains bounded as $y\to\partial D$. Since $R_D^{(0)}$ is real-valued, taking real parts gives
\[
\Re\,R_D^{(k_0)}(y)-R_D^{(0)}(y)=\Re\bigl(R_D^{(k_0)}(y)-R_D^{(0)}(y)\bigr),
\]
and the right-hand side remains bounded as $y\to\partial D$. On the other hand, Lemma~\ref{lem:harmonic-robin-boundary} gives
\[
R_D^{(0)}(y)\to -\infty \quad\text{as } y\to\partial D,\ y\in D.
\]
Therefore
\[
\Re\,R_D^{(k_0)}(y)\to -\infty
\quad\text{as } y\to\partial D,\ y\in D,
\]
and the equivalent statement for $u^s(y,y,k_0)$ follows.
\end{proof}

The following topological lemma describes what happens when two domains are not identical but share a common intersection component.

\begin{lemma}[Boundary crossing for unequal domains]
\label{lem:boundary-crossing}
Let $D_1,D_2\subset\mathbb R^2$ be two bounded connected Lipschitz domains, and let $D^\ast$ be
a connected component of $D_1\cap D_2$. If $D_1\neq D_2$, then there exist
a point
\[
p\in(\partial D_1\cap D_2)\cup(\partial D_2\cap D_1)
\]
and a sequence $(y_n)\subset D^\ast$ with $y_n\to p$.
\end{lemma}

\begin{proof}
If $D_1\subset D_2$ and $D_2\subset D_1$, then $D_1=D_2$; hence, after
relabeling, $D_1\setminus D_2\neq\varnothing$. Since $D_1$ is a connected
open subset of $\mathbb R^2$ and hence path connected, and since
$D^\ast$ is open and nonempty, we may choose a continuous path
$\gamma:[0,1]\to D_1$ with $\gamma(0)\in D^\ast$ and
$\gamma(1)\in D_1\setminus D_2$. Set
\[
t^\ast:=\inf\{t\in[0,1]:\gamma(t)\notin D^\ast\}.
\]
The set $\{t\in[0,1]:\gamma(t)\notin D^\ast\}$ is nonempty, since it
contains $1$, and it does not contain $0$, so $0<t^\ast\le 1$. Then
$p:=\gamma(t^\ast)\in\partial D^\ast\cap D_1$, and by definition of
$t^\ast$ there are $t_n\nearrow t^\ast$ with $\gamma(t_n)\in D^\ast$; set
$y_n:=\gamma(t_n)$. Since $D^\ast\subset D_1\cap D_2\subset D_2$, we have
$y_n\in D_2$ for every $n$, and by continuity of $\gamma$,
$p=\lim_n y_n\in\overline{D_2}$. If $p\in D_2$, then $p\in D_1\cap D_2$,
and because $D_1\cap D_2$ is open, a small ball centered at $p$ would lie
in $D_1\cap D_2$ and meet $D^\ast$, forcing $p\in D^\ast$, a
contradiction. Combining $p\in\overline{D_2}$ with $p\notin D_2$, we
obtain $p\in\partial D_2$; together with $p\in D_1$, this gives
$p\in\partial D_2\cap D_1$. The symmetric case yields a point in
$\partial D_1\cap D_2$.
\end{proof}

\begin{theorem}[Single-frequency open-set uniqueness]
\label{thm:openset-single-frequency}
Let $D_1,D_2\subset\mathbb R^2$ be two bounded connected Lipschitz domains. Let $\mathcal U \subset D_1 \cap D_2$ be a nonempty open set, and fix $k_0>0$ with
$k_0^2\notin \sigma(D_1)\cup\sigma(D_2)$.
Assume that
\[
u_1^s(y,y,k_0)=u_2^s(y,y,k_0), \quad y\in\mathcal U.
\]
Then $D_1=D_2$.
\end{theorem}

\begin{proof}
By Lemma~\ref{lem:diagonal}, for $i=1,2$ the backscattering field
$u_i^s(y,y,k_0)$ is real-analytic with respect to $y\in D_i$. Hence the difference
$u_1^s(y,y,k_0)-u_2^s(y,y,k_0)$ is real-analytic on $D_1\cap D_2$. Choose a nonempty connected open subset $\mathcal U_0\Subset \mathcal U$, and let $D^\ast$ be the connected component of $D_1\cap D_2$ containing $\mathcal U_0$. Since this difference vanishes identically on $\mathcal U_0$, the identity theorem for real-analytic functions implies that
\[
u_1^s(y,y,k_0)=u_2^s(y,y,k_0), \quad y\in D^\ast.
\]

Assume for contradiction that $D_1\neq D_2$. By Lemma~\ref{lem:boundary-crossing}, applied
with the connected component $D^\ast$ of $D_1\cap D_2$ containing
$\mathcal U_0$, there exist a point
\[
p\in (\partial D_1\cap D_2)\cup(\partial D_2\cap D_1)
\]
and a sequence \((y_n)\subset D^\ast\) with \(y_n\to p\).
By symmetry, it suffices to consider the case
\[
p\in \partial D_1\cap D_2.
\]
Since $p\in D_2$, the function $y\mapsto u_2^s(y,y,k_0)$ extends continuously to $p$, and
\[
u_2^s(y,y,k_0)\to u_2^s(p,p,k_0)
\quad\text{as } y\to p,\ y\in D^\ast.
\]
In particular,
\[
\Re\,u_2^s(y,y,k_0)\to \Re\,u_2^s(p,p,k_0)
\quad\text{as } y\to p,\ y\in D^\ast.
\]
On the other hand, Lemma~\ref{lem:diag-regular-part-blowup} gives
\[
\Re\,u_1^s(y,y,k_0)\to -\infty
\quad\text{as } y\to p,\ y\in D^\ast,
\]
because $p\in \partial D_1$. Since the two diagonal values agree on $D^\ast$, it follows that
\[
\Re\,u_2^s(y,y,k_0)\to -\infty
\quad\text{as } y\to p,\ y\in D^\ast,
\]
which contradicts the finite limit $\Re\,u_2^s(p,p,k_0)$. Consequently, \(D_1=D_2\). The proof is complete.
\end{proof}

\section{Discussion and Open Problems} \label{sec:discussion}

We discuss the main limitation of the present approach, its relation to inverse boundary spectral problems, and the open problems left by this work.

\subsection{The multiplicity gap}
As already recorded in Remark~\ref{rmk:simplicity}, the hypotheses in
Theorem~\ref{thm:main-unique}\textup{(b1)} and \textup{(b2)} are conditional. Neither equality of visible
multiplicities nor equality of restriction ranks is presently known to be
inferable from the backscattering data.

To make this obstruction precise, recall two facts. First, by
Definition~\ref{def:visible} and Lemma~\ref{lem:invisible-interior}, the
restriction map \(R_\lambda:E_\lambda(D)\to C(\M)\) splits the
\(\lambda\)-eigenspace into a \emph{visible part} and a \emph{vanishing
part}: its kernel is the subspace of eigenfunctions vanishing on \(\M\),
and its rank counts the independent eigenfunctions that do not. Second,
by Lemma~\ref{lem:visible-spectrum}, the backscattering data determine
only the diagonal values \(P_\lambda(y,y)\) on \(\M\).
These diagonal values do not determine
\(\operatorname{rank}R_\lambda\) for an arbitrary positive semidefinite
kernel, and they do not directly provide the multiplicity \(m_\lambda\).
Whether the PDE constraint changes this is the open question formulated
below.

This obstruction arises already in finite dimensions. For a positive
semidefinite \(m\times m\) matrix \(Q\), the vector of diagonal entries
\((Q_{11},\dots,Q_{mm})\) does not determine \(\operatorname{rank}Q\):
the same diagonal can correspond to different ranks. For example,
\[
\begin{pmatrix}
1 & 0\\
0 & 1
\end{pmatrix}
\quad\text{and}\quad
\begin{pmatrix}
1 & 1\\
1 & 1
\end{pmatrix}
\]
have the same diagonal but ranks \(2\) and \(1\), respectively. The
spectral projection kernel \(P_\lambda(y,z)\) is a positive semidefinite
kernel on \(\M\times\M\), and our data consist only of
its diagonal values \(P_\lambda(y,y)\). The rank of the restricted
kernel is exactly \(\operatorname{rank}R_\lambda\), which cannot in
general be recovered from the diagonal alone.

Lemma~\ref{lem:invisible-interior} controls the total number of
eigenfunction directions invisible from \(\M\), including those
coming from partially visible eigenspaces, by \(N_\Omega(E)\). What
remains uncontrolled is the multiplicity carried by the visible
directions. Whether the PDE-constrained diagonal data force
multiplicity equality, and hence whether the area estimate in
Theorem~\ref{thm:main-unique}\textup{(b1)} and \textup{(b2)} holds without
any multiplicity assumption, remains open. Resolving this would require
either proving that equality of the diagonal values \(P_\lambda(y,y)\)
on \(\M\) forces equality of multiplicities, or introducing
additional data such as the full bistatic kernel \(P_\lambda(y,z)\) for
\(y,z\in\M\).

Two further questions are open. First, we know of no pair of distinct
domains for which the data-equality assumption of
Theorem~\ref{thm:main-unique} is known to hold, nor of any proof that
such a pair cannot exist. In particular, it is open whether the hypotheses of Theorem~\ref{thm:main-unique}\textup{(b1)} and \textup{(b2)} are ever met nontrivially. Second, we do not know whether the constant
$|\Omega|$ is sharp within the class of domain pairs satisfying the
data-equality assumption: equality fails for $D_1=D_2$, no nontrivial
example attaining or nearly attaining the bound is known, and the proof
uses only the counting inequality $|N_{D_1}(E)-N_{D_2}(E)|\le
N_\Omega(E)$. Any
improvement of the counting inequality would have to exploit geometric or
PDE information beyond the diagonal-to-rank comparison used here.

\subsection{Relation to inverse boundary spectral problem and potential theory}

The classical inverse boundary spectral problem clarifies the information missing here. In that setting, one considers the Dirichlet eigenvalues together with the boundary normal derivatives of Dirichlet eigenfunctions, modulo the usual basis ambiguity within each eigenspace. The boundary control method of Belishev \cite{Belishev1997} shows that such full boundary spectral data determine the manifold up to isometry; see also \cite{KKL2001}.

Our interior diagonal data are much weaker: on $\M$ we determine
only the diagonal of the spectral projection kernel, not the boundary
normal-derivative data of the classical problem.

The interior-data results of Bosi, Kurylev, and Lassas
\cite{BosiKurylevLassas2022} and of Kurylev, Lassas, Lu, and Yamaguchi
\cite{KurylevLassasLuYamaguchi2012} likewise use information on an open set,
but start from the full restricted spectral data or the full heat kernel; by
contrast, Theorem~\ref{thm:openset-single-frequency} uses only the Helmholtz
diagonal regular part on \(\mathcal U\), at a single non-resonant
frequency.

The classical potential theory of Robin functions and harmonic radii is
naturally compared with the open-set theorem, since the object involved
is again the diagonal regular part of the Green function. The Robin
function $R_D^{(0)}(y)=H_y(y)$, the trace of the regular part of the
Dirichlet Green function, and the associated harmonic (conformal) radius,
related here by $r_D(y)=\exp\bigl(2\pi R_D^{(0)}(y)\bigr)$ in the
normalization of this paper, are classical objects; see Duren and
Schiffer \cite{DurenSchiffer1991}, Bandle and Flucher
\cite{BandleFlucher1996}, Flucher \cite{Flucher1999}, Cardaliaguet and
Tahraoui \cite{CardaliaguetTahraoui2002}, Gustafsson and Sakai
\cite{GustafssonSakai1994}, and Sakai \cite{Sakai1982}. In particular,
the boundary divergence of the Robin function is a classical theme;
under the present Lipschitz assumptions it is established directly in
this paper. Lemma~\ref{lem:harmonic-robin-boundary} proves the
quantitative logarithmic upper estimate needed for the argument and, in
particular,
\[
R_D^{(0)}(y)\to-\infty\quad\text{as }y\to\partial D,\ y\in D.
\]
Lemma~\ref{lem:diag-regular-part-comparison} then transfers this
divergence to the shifted Helmholtz regular part, since
$R_D^{(k_0)}-R_D^{(0)}$ remains bounded up to the boundary. What the
classical literature does \emph{not} provide directly is the inverse
statement in the precise form required here, and we record the
differences as follows:
\begin{enumerate}
\item[(a)] \emph{Regularity.} The classical results on the Robin
function, the harmonic radius, and the harmonic center are formulated
for smooth (typically $C^2$, often convex or analytic) boundaries,
whereas we work on bounded Lipschitz domains and rely on
harmonic-measure estimates through the NTA machinery of Jerison and
Kenig \cite{JerisonKenig1982}.
\item[(b)] \emph{Frequency.} Our datum is the diagonal regular part of
the Green function of the \emph{shifted} Helmholtz operator
$-\Delta-k_0^2$ at a fixed nonzero frequency $k_0$, whereas the
classical Robin function corresponds to $k_0=0$.
\item[(c)] \emph{Data and conclusion.} We are given the diagonal values
only on a \emph{nonempty open subset} $\mathcal U$ of $D_1\cap D_2$, not
the full function on a domain. The analytic-continuation step of
Theorem~\ref{thm:openset-single-frequency} upgrades this local
information to a global statement, and the boundary blow-up then yields
$D_1=D_2$.
\end{enumerate}

Even complete Dirichlet spectral data, namely all eigenvalues with their multiplicities, do not determine a planar domain; see Gordon, Webb, and Wolpert \cite{GordonWebbWolpert1992}. This reflects a limitation of spectral information alone, rather than a nonuniqueness mechanism for the present curve-based diagonal data. 

\subsection{Open problems}
Based on the previous discussion, we propose the following open problems. These problems are not of merely theoretical interest: backscattering
measurement configurations are widely used in practice, and each item
below isolates a specific obstruction to a sharper uniqueness theory
for such data.

\begin{enumerate}
\item \emph{The multiplicity gap, nontrivial examples, and sharpness.}
As discussed above, it is open whether the diagonal data force equality
of multiplicities or of restriction ranks; whether any pair of distinct
domains with identical curve data exists at all; and whether the constant
$|\Omega|$ in the area bound is sharp.
\item \emph{Finite-frequency data.} Whether additional a priori structure could make finitely many frequencies suffice, replacing the accumulation-point requirement of the residue argument, is open.
\item \emph{Neumann and impedance boundary conditions.} The paper treats the
Dirichlet problem throughout. Since Qin and Colton \cite{QinColton2012acm}
study an impedance boundary condition in the interior cavity setting, it is
natural to ask whether the present mechanism survives, i.e., the boundary
blow-up mechanism of the single-frequency theorem would have to be rebuilt
for the corresponding Green function.
\item \emph{Stability.} The uniqueness results in both settings are
qualitative. In the curve-measurement setting, the continuation in the
frequency variable is used only as an exact-data determinacy device: it
requires agreement on a set of frequencies with an accumulation point, and
no stability of the analytic continuation is claimed. In the open-set setting, if
$u_1^s(\cdot,\cdot,k_0)-u_2^s(\cdot,\cdot,k_0)$ is small in
$L^\infty(\mathcal U)$, one expects the two domains to be close in a
Hausdorff-type sense, but no estimate follows from the arguments here, since
both the analytic-continuation step and the boundary-blow-up step are
inherently non-quantitative. Quantitative stability in either setting would
substantially strengthen the results and appears to require a different
approach.
\end{enumerate}

\section{Conclusion} \label{sec:conclusion}
This paper addresses a natural but largely unexplored question in interior
inverse scattering: what do coincident source-receiver data, measured
inside the unknown domain itself, actually determine? Backscattering measurements are among the most common configurations in practice, yet the corresponding
uniqueness theory lags far behind that for full multistatic data.

We treat this question in two complementary regimes, each tied to a
different measurement model.

In the curve-measurement setting, the meromorphic frequency dependence of
the backscattering field yields spectral information: the residues at the
poles recover the diagonal of the spectral projection kernel on the
measurement curve, hence the visible part of the Dirichlet spectrum. This
yields an unconditional comparison of the two Dirichlet spectra, with the
symmetric difference contained in the spectrum of the enclosed domain. An
area estimate, controlled by that same domain, holds under either matching
multiplicities or matching restriction ranks of the common visible
eigenvalues. 

In the open-set setting, a single non-resonant frequency identifies the
diagonal regular part of the Helmholtz Green function on the observation
set. Its interior real-analyticity propagates the identity throughout
the relevant intersection component, and its boundary blow-up forces the
two domains to coincide, yielding a scattering-theoretic uniqueness
theorem for bounded connected Lipschitz domains. 

The two regimes concern the same object, namely the
diagonal of the Green function in its spectral or pointwise form, but they
reveal different aspects of it: the curve setting extracts global spectral
information yet stops at the multiplicity gap, while the open-set setting
gives full uniqueness from a single frequency but requires a
two-dimensional observation set. Closing the gap between them remains
the central challenge for the information content of backscattering measurements.

\section*{Acknowledgments}
The research of X.~Liu is supported by the National Key R\&D Program of China under grant 2024YFA1012303 and by the National Natural Science Foundation of China under grant 12371430.

\end{document}